\documentclass[12pt,a4paper,reqno]{amsart}

\usepackage{mathrsfs}

\usepackage{amsmath}   
\usepackage{amsfonts}  

\makeatletter
\def\subsection{\@startsection{subsection}{2}%
  \z@{.5\linespacing\@plus.7\linespacing}{.3\linespacing}%
  {\normalfont\bfseries}}
\makeatother

\usepackage{tikz}
\usetikzlibrary{positioning} 
\usetikzlibrary{arrows.meta} 
\usetikzlibrary{automata, arrows, positioning}
\usepackage[normalem]{ulem}
\usetikzlibrary{positioning, arrows.meta, fit}  

\newcommand{\ignore}[1]{}

\usepackage{pgfplots}
\pgfplotsset{compat=1.17}

\usetikzlibrary{positioning, arrows.meta, shapes.geometric, calc}
\usepackage{amsmath, amsthm, amssymb}
\usepackage[T1]{fontenc}
\usepackage{tgbonum}
\usepackage{authblk}
\usepackage{float}
\usepackage{subcaption}
\usepackage{caption}
\usepackage{graphicx} 
\usepackage{times,authblk}
\usepackage{amsmath, amsfonts, amssymb, amsthm, bm, stmaryrd, enumerate, mathtools}
\usepackage{xspace}
\usepackage{tikz}
\usepackage{comment}
\usepackage{mathabx}
\usetikzlibrary{arrows.meta}
\usepackage{enumitem}
\usepackage{tikz-cd}
\usepackage{amsmath, amsthm} 
\usepackage[colorlinks=true,linkcolor=blue]{hyperref} 
\tikzset{commutative diagrams/.cd}
\tikzcdset{row sep/tiny=0.12 em}

\usetikzlibrary{positioning,arrows.meta}
\usetikzlibrary{positioning}
\usetikzlibrary{cd}

\usetikzlibrary{positioning,arrows,fit}

\usetikzlibrary{intersections}

\newlength\myheight

        \usetikzlibrary{calc}

\usepackage{amsthm}

\newtheorem{Def}{Definition}[section]
\newenvironment{definition}{\begin{Def} \rm}{\end{Def}}

\newtheorem{example}[Def]{Example}

\theoremstyle{definition} 
\newtheorem{lemma}[Def]{Lemma}
\newtheorem{proposition}[Def]{Proposition}

\newtheorem{theorem}[Def]{Theorem}
\newtheorem{remark}[Def]{Remark}

\newtheorem{claim}{Claim}[Def]

\newcommand{\leqs}{\leqslant}

\newcommand{\FSUPH}{{\mbox{$F\kern -.85ex-\kern -.35ex\mathbb{S}$}}}
\newcommand{\FSUPL}{{\mbox{$F\kern -.85ex-\kern -.35ex\mathbb{S}_\leqs$}}}
\newcommand{\smFSUPL}{{{F\kern -.35ex-\kern -.08ex\mathbb{S}_\leqs}}}
\newcommand{\FINFH}{{\mbox{$F\kern -.85ex-\kern -.35ex\mathbb{M}$}}}
\newcommand{\FINFL}{{\mbox{$F\kern -.85ex-\kern -.35ex\mathbb{M}_\leqs$}}}
\newcommand{\VSUPH}{{\mbox{$V\kern -.85ex-\kern -.35ex\mathbb{S}$}}}
\newcommand{\VJ}{{\mbox{$V\kern -.85ex-\kern -.35ex\mathbb{J}$}}}
\newcommand{\VMod}{{\mbox{$V\kern -.85ex-\kern -.35ex\mathbf{Mod}$}}}
\newcommand{\VFSUPH}{{\mbox{$V\kern -.85ex-\kern -.35ex F\kern -.85ex-\kern -.35ex\mathbb{S}$}}}
\newcommand{\VFSUPL}{{\mbox{$V\kern -.85ex-\kern -.35ex F\kern -.85ex-\kern -.35ex\mathbb{S}_\leqs$}}}

\newcommand{\EndoHom}[1]{\mathcal{E}{\kern-.5ex}\textit{n}{\kern-.2ex}\textit{d}{\kern-.2ex}\textit{o}(#1)}

\usepackage{graphicx} 

\usepackage{enumerate,amsmath,amsthm,amsfonts}
\usepackage{amsmath}
\usepackage{amsfonts}
\usepackage{amssymb}

\usepackage{epsfig}
\usepackage{graphicx,color}
\usepackage{graphpap,color}
\usepackage{graphicx}
\usepackage{latexsym}
\usepackage{float}
\usepackage{setspace}
\usepackage{hyperref}   

.360pk
.360pk
.360pk

\usepackage{amsfonts,bm}

\usepackage[utf8]{inputenc}
\usepackage[T1]{fontenc}

\usepackage{mathrsfs}  
\usepackage{amssymb} 
\usepackage{mathtools}
\usepackage{tikz} 

\usepackage{geometry}

\begin{document}
\include {mak}
~\vspace{-16mm}


\newcommand{\norm}[1]{\left\|#1\right\|}
\newcommand{\hkd}[1]{\left\langle#1\right\rangle}
\newcommand{\krk}[1]{\left\{#1\right\}}
\newcommand{\krb}[1]{\left(#1\right)}
\newcommand{\abs}[1]{\left|#1\right|}
\newcommand{\ol}[1]{\overline{#1}}
\newcommand{\rref}[1]{[\ref{#1}]}

\oddsidemargin 16.5truemm
\evensidemargin 16.5truemm

\thispagestyle{plain}

\begin{center}
{\Large\bf CATEGORICAL EQUIVALENCE BETWEEN FINITARY ORTHOMODULAR DYNAMIC ALGEBRAS AND ORTHOMODULAR LATTICES\\ 
\vspace{1.5cc}
{\large\sc JUANDA KELANA PUTRA$^{1,2}$ AND RICHARD SMOLKA$^{1}$}\\

\vspace{0.3 cm} {\small $^{1}$ Department of Mathematics and Statistics, Masaryk University, Czech Republic}
\\
{\small $^{2}$ Department of Mathematics, Walisongo State Islamic University, Indonesia}

\rule{0mm}{6mm}
\begin{NoHyper}
\renewcommand{\thefootnote}{}\footnotetext{  
\\
{\rm Received: dd-mm-yyyy, accepted: dd-mm-yyyy.}
}
\end{NoHyper}
}

\vspace{1.5cc}

\parbox{24cc}{{\Small{\bf Abstract.}
This paper reveals a categorical equivalence connecting two distinct quantum logic structures. The first is the orthomodular lattice, an algebraic system designed to formalize the properties of quantum systems. The second is a finitary orthomodular dynamic algebra, a specialized development of the orthomodular dynamic algebra where the underlying quantum actions are restricted to be finitary. The applicability of the result extends to more specialized lattices, such as Hilbert lattices of closed subspaces of a Hilbert space, beyond general orthomodular lattices. As these lattice structures exhibit connections to a diverse array of quantum structures, the established equivalence categorically bridges unital involutive m-semilattices with a broad spectrum of quantum formalisms.

}}
\end{center}

\vspace{0.25cc}
\parbox{24cc}{\Small {\it Key words and Phrases}: Finitary orthomodular dynamic algebra, Orthomodular lattice, Unital involutive m-semilattice}

\vspace{1.5cc}

\section{Introduction}\label{Orthomodular dynamic algebras}
Various quantum logic structures prioritize different elements of quantum reasoning. Our focus in this paper is on two particular quantum logic structures. The first structure we examine is orthomodular lattice, which serve as the core algebraic structure underpinning quantum logic, providing a formal, non-classical framework for modeling the testable properties or propositions about a quantum system. The second structure we consider is the orthomodular dynamic algebras, which provide an axiomatic framework that makes quantum actions first-class citizens, allowing for formal reasoning about their composition and quantum-logical properties within a unified system. It extends the static view of quantum properties (captured by orthomodular lattices) to include the dynamic, operational aspects of quantum systems, as proposed in \cite{KRSZ}.

This paper aims to establish a categorical equivalence between orthomodular lattices and finitary orthomodular dynamic algebras, a specialized form of orthomodular dynamic algebra. This equivalence will clarify how dynamics emerges from orthomodular or Hilbert lattices  of closed subspaces of a Hilbert space, and solidify the connection between unital involutive m-semilattices (please refer to \cite{Botur} for a detailed survey of unital involutive m-semilattices) and orthomodular lattices, ultimately enhancing their application in quantum logic phenomena. 

The most significant finding about the connection between orthomodular lattices and orthomodular dynamic algebras is presented in \cite{KRSZ}. That work aimed to prove a categorical equivalence between complete orthomodular lattices (representing static quantum properties) and orthomodular dynamic algebras (modeling dynamic quantum actions). This equivalence highlights the mathematical interconnectedness of static and dynamic quantum perspectives, clarifying how dynamics originates from orthomodular lattices and strengthening the link between dynamic algebras and quantum reasoning.

Extensively used to relate different structures and transfer results, categorical equivalence and duality are vital. An equivalence comprises two functors whose compositions are naturally isomorphic to identity functors, while a duality is an equivalence with reversed morphisms in one category. Past research has established quantum-significant categorical equivalences (e.g., between quantum geometries and lattices, as in \cite{StubbeSteirteghem2007}) and dualities (e.g., between quantum lattices and graph-like structures, as in \cite{BergfeldKishidaSackZhong2014}). Our proposed equivalence builds upon this by connecting unital involutive m-semilattices to these established structures, specifically through their relation to propositional systems and Hilbert lattices.

The paper proceeds as follows: We first introduce the categories of orthomodular lattices and finitary orthomodular dynamic algebras in Section 2. Section 3 then details the functors comprising the categorical equivalence, explaining how they translate objects and map morphisms between the categories. Subsequently, Section 4 demonstrates this equivalence by constructing natural isomorphisms between the composed functors and the identity functors. The paper concludes in Section 5 with final remarks and directions for future research.

\section{Preliminaries}

\subsection{The Category $\mathbb{OML}$ of Orthomodular Lattices}

Our initial step involves standardizing notation and revisiting the foundational structures. An orthomodular lattice is defined as a bounded lattice augmented by an appropriate orthocomplement. For additional properties of the orthomodular lattices, we refer to \cite{Kalmbach83}.


\begin{definition}
An \textit{ortho-lattice} is a tuple $\mathcal{M} = \left( {M, \le ,\mathop  - \nolimits^ \bot} \right)$ satisfying the following conditions:
\begin{enumerate}
  \item $\left( {M, \le} \right)$ is a bounded lattice with least element $0$ and greatest element $1$;
  \item  The map $\mathop  - \nolimits^ \bot  :M \to M$ satisfies, for all $m,n\in M$,
    \begin{enumerate}
    \renewcommand{\labelenumi}{\Alph{enumi}.}
    
    \item  $m\wedge m^{\perp}=0$ and $m\vee m^{\perp}=1$,
    
    \item  $m \le n \Rightarrow {n^ \bot } \le {m^ \bot }$; 

    \item  $(m^{\perp})^{\perp}=m$,
\end{enumerate}
\end{enumerate}
An ortho-lattice $\mathcal{M}$ is defined to be an \textit{orthomodular lattice} if for all $m,n \in Y$ such that $m\le n$, it holds that $ n = m \vee \bigl(m^{\perp}\wedge n\bigr)$.
\end{definition}


Subsequently, we will provide a definition for the category of orthomodular lattices and ortho-lattice isomorphisms.


\begin{definition}
Given two orthomodular lattices ${\mathcal{M}_1} = \left( {{M_1},{ \le _1},\mathop  - \nolimits^{{ \bot _1}} } \right)$ and ${\mathcal{M}_2} = \left( {{M_2},{ \le _2},\mathop  - \nolimits^{{ \bot _2}} } \right)$, an \textit{ortho-lattice isomorphism} $g:{\mathcal{M}_1} \to {\mathcal{M}_2}$ is a function $g:{M_1} \to {M_2}$ that satisfies the following conditions:
\begin{enumerate}
\item $g$ is a bijection;
\item ${m} \le_1 {n} \Leftrightarrow g\left( {{m}} \right) \le_2 g\left( {{n}} \right)$;
\item $g\left( {{m^{{ \bot _1}}}} \right) = {\left( {g\left( m \right)} \right)^{^{{ \bot _2}}}}$,
\end{enumerate}
for every ${m},{n} \in {M_1}$.\\
It can easily be shown that $\mathbb{OML}$ forms a category where objects are orthomodular lattices and morphisms are ortho-lattice isomorphisms.
\end{definition}



\begin{remark}
Let $g:{\mathcal{M}_1} \to {\mathcal{M}_2}$ be an ortho-lattice isomorphism. Then for all \(m,n\in M_1\), we have $g(m \vee n) = g(m)\vee g(n)$ and $ g(m\wedge n) = g(m) \wedge g(n) $.
\end{remark}

Additionally, within an ortho-lattice $\mathcal{M} = (M, \le, {}^{\perp})$, we can define two key operations for each element $m \in M$:
\begin{enumerate}
    \item \textit{Sasaki projection} (onto $m$) : This map, denoted ${\pi _m}: M \to M$ takes an element $n$ to $m \land (m^\perp \lor n)$;

    \item \textit{Sasaki hook} (from $m$) : This map, denoted $\pi^m: M \to M$ takes an element $n$ to $m^\perp \lor (m \land n)$.
\end{enumerate}
These two maps are always order-preserving. A significant property is that the ortho-lattice $\mathcal{M}$ is orthomodular if and only if, for every $m \in M$, the Sasaki projection $\pi _m$ is the left adjoint of the Sasaki hook $\pi^m$. This means that in an orthomodular lattice, Sasaki projections preserve arbitrary joins.



\begin{lemma}
Let $\mathcal{M} = \left( {M, \le ,\mathop  - \nolimits^ \bot} \right)$ be an orthomodular lattice. If a subset $A\subseteq M$ possesses a supremum (i.e., $\mathop \bigvee\limits_{a \in A} a$ exists), then for every $m \in M$, we have:
\begin{center}
$\pi _m\!\left(\mathop \bigvee\limits_{a \in A} a\right)
\;=\;\mathop \bigvee\limits_{a \in A} \pi_m(a).$
\end{center}
\end{lemma}

\begin{lemma}\label{Sasaki is Involutive}
 Let $\mathcal{M} = \left( {M, \le ,\mathop - \nolimits^ \bot} \right)$ be an orthomodular lattice. Define the set $S_\mathcal{M}$ as the set of all finite compositions of Sasaki projections on $M$:
$$S_\mathcal{M} = \left\{ {{\pi _{{m_1}}} \circ \cdots \circ {\pi _{{m_j}}} \mid {m_1}, \ldots ,{m_j} \in M, j \in {\mathbb{N}^+}} \right\}$$
This set $S_\mathcal{M}$ is precisely the smallest set that contains all Sasaki projections $\{\pi_m : M\to M \mid m\in M\}$ and is closed under the operation of function composition ($\circ$).

Furthermore, an involution ${ - ^ * }:S_\mathcal{M} \to S_\mathcal{M}$ is defined for any element $P = {\pi _{{m_1}}} \circ \cdots \circ {\pi _{{m_j}}} \in S_\mathcal{M}$ as:
$$P^ * = {\left( {{\pi _{{m_1}}} \circ \cdots \circ {\pi _{{m_j}}}} \right)^ * } = \pi _{{m_j}}^ * \circ \cdots \circ \pi _{{m_1}}^ * $$
Based on these definitions, the algebraic structure $\left(S_\mathcal{M} ,\circ, \operatorname{id}_M,{- ^ *}\right)$ constitutes an involutive monoid, where $\operatorname{id}_M$ denotes the identity map on $M$.
\end{lemma}


\subsection{The Category $\mathbb{FODA}$ of Finitary Orthomodular Dynamic Algebras}

This subsection establishes the category of finitary orthomodular dynamic algebras. Essentially, a finitary orthomodular dynamic algebra enriches an $m$-semilattice. We will start by defining the $m$-semilattice and other necessary notions, then proceed to define the quantum dynamic algebra. Further information regarding $m$-semilattices can be obtained from \cite{Paseka1986} and \cite{Botur}.

\begin{definition}
An \textit{m-semilattice} is a tuple $\mathcal{K} = (K, \sqcup,  \odot )$ satisfying the following conditions:
\begin{enumerate}
    \item $(K, \sqcup)$ is a bounded join-semilattice, where $0 \in K$ is the least element and $1 \in K$ is the greatest element with respect to the partial order induced by the join operation $\sqcup$.
    \item The binary operation $\odot : K \times K \to K$ is associative.
    \item The operation $\cdot$ distributes over finite joins in $K$. Specifically, for every $k \in K$ and for any finite indexed collection $\{ {k_i} \mid i \in J\} $ of elements from $K$, the following properties hold:
    \begin{enumerate}[label=\alph*.]
\item $k \odot \left( {\mathop \bigsqcup \limits_{i \in J} {k_i}} \right) = \mathop \bigsqcup \limits_{i \in J} \left( {k \odot {k_i}} \right)$;

\item $\left( {\mathop \bigsqcup\limits_{i \in J} {k_i}} \right) \odot k = \mathop \bigsqcup\limits_{i \in J} \left( {{k_i} \odot k} \right)$.
\end{enumerate}
\end{enumerate}
\end{definition}

The partial order induced by the join operation on $K$ is denoted by $\sqsubseteq$. A direct consequence of the distributivity of the binary operation $\odot$ over finite joins is that $\odot$ is monotone with respect to $\sqsubseteq$. Furthermore, this distributivity implies that the element $0$ acts as a zero element within the semigroup $(K, \odot)$.

An \textit{involutive m-semilattice} is formally defined as an $m$-semilattice $K$ equipped with a unary operation denoted by $\mathop  - \nolimits^ *  $, which serves as a semigroup involution. This operation satisfies the following properties for all $k, l \in K$ and for any finite indexed collection $\{ {k_i}:i \in J\} $ of elements from $K$:
\begin{enumerate}
    \item $k^{**} = k$
    \item $(k \odot l)^* = l^* \odot k^*$
    \item ${\left( {\mathop \bigsqcup \limits_{i \in J} {k_i}} \right)^ * } = \mathop \bigsqcup \limits_{i \in J} k_i^ * $
\end{enumerate}
An involutive $m$-semilattice is called \textit{unital} iff there exists \(e \in K\) such that \(e \cdot a = a \cdot e = a\) for all \(a \in K\).

We now introduce the finitary generalized dynamic algebra, an extended form of the generalized dynamic algebra. For background on generalized dynamic algebras, see \cite{KRSZ}.

\begin{definition}\label{finitary-goda}
A \emph{finitary generalized dynamic algebra} is a tuple \(\mathfrak{K} = (K, \bigsqcup, \odot, \thicksim)\) satisfying the following conditions:
\begin{enumerate}
    \item \(K\) is a non-empty set.
    \item \(\bigsqcup : \mathscr{P}_{\text{fin}}(K) \rightarrow K\), where \(\mathscr{P}_{\text{fin}}(K)\) denotes the set of all finite subsets of \(K\), is a finitary operation on \(K\).
    \item \(\odot : K \times K \to K\) is a binary operation on \(K\).
    \item \(\thicksim : K \to K\) is a unary operation on \(K\).
\end{enumerate}
\end{definition}

We now present several constructions involving a finitary generalized dynamic algebra \(\mathfrak{K} = (K, \bigsqcup, \odot, \thicksim)\).

\begin{equation*}
\begin{split}
\widetilde K & \stackrel{\text{def}}{=} \{ \thicksim k \mid k \in K \} \\
\bigvee W & \stackrel{\text{def}}{=} \,\, \thicksim \left(\thicksim \bigsqcup W\right), \, \text{for any finite}\, W \subseteq \widetilde K\\
\bigwedge W & \stackrel{\text{def}}{=} \,\, \thicksim \bigsqcup \left\{ { \thicksim w:w \in W} \right\}
, \, \text{for any finite}\, W \subseteq \widetilde K\\
\preceq \,\,\,\,\, & \stackrel{\text{def}}{=} \,\, \left\{ {\left( {k,l} \right) \in \widetilde K \times \widetilde K \mid \bigvee \left\{ {k,l} \right\} = l} \right\}\\
\left\langle \widetilde K \right\rangle  & \stackrel{\text{def}}{=} \,\, \left\{ {k \in K:k = {w_1}  \odot  \cdots  \odot {w_n}, \, \text{for some}\, n \in {\mathbb{N}^+} \, \text{and}\, {w_1}, \ldots ,{w_n} \in \widetilde K} \right\}.
\end{split}
\end{equation*}

For a fixed element ${k \in K}$, we define a unary operation ${\ulcorner k \urcorner : K \to K}$. This operation is formally given by:
$$ \ulcorner k \urcorner \left( l \right) =\,\, \thicksim \left( { \thicksim \left( {k \odot l} \right)} \right) $$
for every ${l \in K}$. Building on this, we introduce an equivalence relation, denoted by ${\equiv}$:
$$ \equiv \,\,\stackrel{\text{def}}{=} \,\, \left\{ {\left( {k,l} \right) \in K \times K \mid \ulcorner k \urcorner \left( w \right) = \ulcorner l \urcorner \left( w \right), \, \text{for every}\, w \in \widetilde K} \right\}. $$
Two elements $k$ and $l$ are identified by $ \equiv$ iff they produce the same result when their corresponding unary operations are applied to any ${w \in \widetilde K}$.

This framework then allows us to introduce the finitary orthomodular dynamic algebra, which expands upon the existing structure of orthomodular dynamic algebras. For a more detailed understanding of orthomodular dynamic algebras, please refer to \cite{KRSZ}.

\begin{definition}\label{deffoda}
A \textit{finitary orthomodular dynamic algebra} is a finitary generalized dynamic algebra $\mathfrak{K} = (K , {\bigsqcup } , {\odot } , \thicksim )$ extended with a unary operation ${-^*}:K \to K$. This extended structure must satisfy the following conditions:

\begin{enumerate}[label=(\textbf{FODA \arabic*}), leftmargin=2cm]
    \item $(K, \sqcup, \odot, {-^*})$ forms a unital involutive $m$-semilattice. Here, $\sqsubseteq$ is the partial order defined by $\sqcup$, and $\bigsqcup$ represents a finite join with respect to $\sqsubseteq$.

    \item $(\widetilde K , {\preceq } , \thicksim )$ is an orthomodular lattice. An essential property is that if an element $x$ is in $\widetilde K$, then its conjugate $x^*$ must also be in $\widetilde K$.

    \item If a set $A$ meets these criteria:
    \begin{enumerate}[label=(\alph*), leftmargin=0.69cm]
        \item $\widetilde K \subseteq A \subseteq K$.
        \item $A$ is closed under both the $\odot$ and ${-^*}$ operations.
        \item $A$ is closed under $\bigsqcup$, meaning that for any finite subset $B$ of $A$,
        its join ($\bigsqcup B$) must also be in $A$.
    \end{enumerate}
    Then, $A$ must be equal to $Q$. This condition ensures minimality.

    \item For any subsets $S , T \subseteq \left\langle \widetilde K \right\rangle$, their joins are equal ($\bigsqcup S = \bigsqcup T$) if and only if the sets themselves are equal ($S = T$). This ensures proper set equality.

    \item For any ${s,t\in \left\langle \widetilde K \right\rangle}$, $s = t$ if and only if $s \equiv t$. This condition guarantees completeness.

    \item For any $v , w \in \widetilde K $, the unary operation $\ulcorner v \urcorner (w)$ is equivalent to the Sasaki projection $\pi_{v}(w)$, which is defined as $\thicksim(\thicksim (v \odot w)) = v \wedge (\thicksim v \vee w)$.

    \item For each ${k, l\in K}$, the composition property holds: $\ulcorner k \urcorner (l) = \ulcorner k \urcorner (\thicksim (\thicksim l))$.
\end{enumerate}
\end{definition}

Our next step is to prove two useful lemmas. The first demonstrates the existence of a normal form for each element in a finitary orthomodular dynamic algebra.
\begin{lemma}\label{lemma2}
Consider a finitary orthomodular dynamic algebra $\mathfrak{K} = (K , {\bigsqcup } , {\odot } , \thicksim )$. For any element \(k\in K\), there is a unique finite set $S \subseteq \left\langle \widetilde K \right\rangle$ such that $k = \bigsqcup\limits_{s \in S} s $.
\end{lemma}

\begin{proof}
Let $A$ be the set of all elements in $K$ that can be expressed as a finite join of elements from $\left\langle \widetilde K \right\rangle$. That is, $A = \left\{ k \in K : k = \bigsqcup\limits_{t \in T} t \text{ for some finite set } T \subseteq \left\langle \widetilde K \right\rangle \right\}$. To show that $A$ satisfies the criteria of \textbf{FODA 3}, consider arbitrary $k, l \in A$, where $k = \bigsqcup\limits_{t \in T} t$ for a finite $T \subseteq \left\langle \widetilde K \right\rangle$, and $l = \bigsqcup\limits_{s \in S} s$ for a finite $S \subseteq \left\langle \widetilde K \right\rangle$.

\begin{enumerate}
\renewcommand{\labelenumi}{\alph{enumi}.}
    \item For any $v \in \widetilde K$, we have $v = \bigsqcup\limits_{x \in \{v\}} x$. Since $\{v\}$ is a finite subset of $\left\langle \widetilde K \right\rangle$, it follows that $v \in A$. Therefore, $\widetilde K \subseteq A \subseteq K$.

    \item Closure under operations:
    \begin{itemize}
        \item For $\odot$:
        \[
        k \odot l = \left(\bigsqcup\limits_{t \in T} t \right) \odot \left(\bigsqcup\limits_{s \in S} s \right) = \bigsqcup\limits_{(t,s) \in T \times S} (t \odot s)
        \]
        Since $T \times S$ is finite and $\{t \odot s : (t,s) \in T \times S\} \subseteq \left\langle \widetilde K \right\rangle$, it implies that $k \odot l \in A$. Thus, $A$ is closed under the $\odot$ operation.

        \item For ${-^*}$ :
        \[
        k^* = \left(\bigsqcup_{t \in T} t \right)^* = \bigsqcup_{t \in T} (t^*)
        \]
        As $t^* \in \left\langle \widetilde K \right\rangle$ for all $t \in T$, and $T$ is finite, we conclude that $k^* \in A$. So, $A$ is closed under the ${-^*}$  operation.
    \end{itemize}

    \item Closure under finite join ($\bigsqcup$):
    \[
    k \sqcup l = \left(\bigsqcup_{t \in T} t \right) \sqcup \left(\bigsqcup_{s \in S} s \right) = \bigsqcup_{u \in T \cup S} u
    \]
    Since $T \cup S$ is a finite subset of $\left\langle \widetilde K \right\rangle$, it follows that $k \sqcup l \in A$. Hence, $A$ is closed under finite joins.
\end{enumerate}
From these three facts, we conclude that $A$ satisfies the criteria of \textbf{FODA 3}. Consequently, $A = K$.

For the uniqueness property: If an element $k$ can be expressed as two different finite joins, say $k = \bigsqcup\limits_{t \in T} t = \bigsqcup\limits_{u \in U} u$ for finite sets $T, U \subseteq \left\langle \widetilde K \right\rangle$, then by \textbf{FODA 4}, we immediately have $T = U$. This establishes the uniqueness of such a representation.
\end{proof}

We now show that the interpretation of a composite test \( p_1 \cdots p_n \in\mathcal P \) can be decomposed into a composition of the interpretations of its components. This observation aligns the dynamic interpretation with Sasaki composition on the lattice of projections:

\begin{lemma}\label{composition}
In a finitary orthomodular dynamic algebra $\mathfrak{K} = (K , {\bigsqcup } , {\odot } , \thicksim, {-^*} )$, the quotation map $\ulcorner - \urcorner$ behaves like a homomorphism, linking concatenation to composition. This means for any $n \in \mathbb{N}^+$ and $w_1,\dots,w_n \in \widetilde K$, we get:
$$ \ulcorner w_1 \odot \cdots \odot w_n\urcorner = \ulcorner w_1\urcorner\circ\cdots\circ\ulcorner w_n\urcorner.$$
This homomorphic property also applies to the product of Sasaki projections, so:
$$ \ulcorner w_1  \odot \cdots \odot  w_n\urcorner = \ulcorner w_1\urcorner\circ\cdots\circ\ulcorner w_n\urcorner = \pi_{p_1}\circ\cdots\circ \pi_{p_n}. $$
\end{lemma}

\begin{proof}
We proceed by mathematical induction.

\noindent \textit{Base Case (n=1)}:
For \(n=1\), we have $\ulcorner w_1 \urcorner = \ulcorner w_1 \urcorner$, which is trivially true.

\noindent \textit{Inductive Hypothesis}:
Assume the claim holds for some \(n=i\), so that for any \(w_1, \dots, w_i \in \widetilde K\):
\[
\ulcorner w_1 \odot \cdots \odot  w_i \urcorner = \ulcorner w_1 \urcorner \circ \cdots \circ \ulcorner w_i \urcorner.
\]

\noindent \textit{Inductive Step}:
We now show that the claim holds for \(n=i+1\). We want to prove that:
\[
\ulcorner w_1  \odot \cdots \odot  w_i  \odot w_{i+1} \urcorner = \ulcorner w_1 \urcorner \circ \cdots \circ \ulcorner w_i \urcorner \circ \ulcorner w_{i+1} \urcorner.
\]
For any \(k \in K\), by definition of \(\ulcorner \cdot \urcorner\) we have:
$$ \ulcorner w_1  \odot \cdots \odot  w_i \odot w_{i+1} \urcorner (k) = \sim\sim \Bigl( (w_1  \odot \cdots \odot  w_i   \odot  w_{i+1}) \odot k \Bigr) $$
Using the associativity of \(\odot\), we regroup as follows:
$$ (w_1 \odot \cdots \odot w_i \odot w_{i+1}) \odot k = w_1 \cdot \Bigl( (w_2  \odot \cdots \odot w_i \odot w_{i+1}) \odot k \Bigr) $$
Thus,
$$ \ulcorner w_1 \odot \cdots \odot w_i \odot w_{i+1} \urcorner (k) = \sim\sim \Bigl( w_1 \odot \bigl((w_2 \odot \cdots \odot w_i \odot w_{i+1}) \odot k\bigr) \Bigr) $$
Now, define $ l := (w_2 \odot \cdots \odot w_i  \odot w_{i+1}) \odot k$. Then, by the definition of \(\ulcorner w_1 \urcorner\):
$$ \ulcorner w_1 \urcorner (l) = \sim\sim (w_1 \odot l) = \sim\sim \Bigl( w_1 \odot \bigl((w_2 \odot \cdots \odot w_i \odot w_{i+1}) \cdot k\bigr) \Bigr) $$
That is, $ \ulcorner w_1 \odot \cdots \odot w_i \odot w_{i+1} \urcorner (k) = \ulcorner w_1 \urcorner (l)$.

Next, from \textbf{FODA 7}, we have that for any \(m, n \in K\), $ \ulcorner m \urcorner (n) = \ulcorner m \urcorner (\sim\sim n)$. Applying this with \(m = w_1\) and \(n = l\), we have $\ulcorner w_1 \urcorner (l) = \ulcorner w_1 \urcorner \Bigl(\sim\sim l \Bigr)$.
Recalling that \(l = (w_2 \odot \cdots \odot w_i \odot w_{i+1}) \odot k\), this becomes:
$$ \ulcorner w_1 \urcorner \Bigl(\sim\sim\bigl((w_2 \odot \cdots \odot w_i  \odot w_{i+1}) \odot k \bigr)\Bigr) $$
Thus, we obtain the rewriting:
$$ \ulcorner w_1 \odot \cdots \odot w_i \odot w_{i+1} \urcorner (k) = \ulcorner w_1 \urcorner \Bigl(\sim\sim\bigl((w_2 \odot \cdots \odot  w_i \odot w_{i+1}) \odot k \bigr)\Bigr) $$
Since \(k\) was arbitrary, it follows that:
$$ \ulcorner w_1 \odot \cdots \odot w_i  \odot w_{i+1} \urcorner = \ulcorner w_1 \urcorner \circ \ulcorner w_2 \cdots w_i w_{i+1} \urcorner $$
By the inductive hypothesis, $\ulcorner w_2 \odot \cdots \odot w_i \odot w_{i+1} \urcorner = \ulcorner w_2 \urcorner \circ \cdots \circ \ulcorner w_{i+1} \urcorner$.
Hence,
$$ \ulcorner w_1 \odot \cdots \odot w_i \odot w_{i+1} \urcorner = \ulcorner w_1 \urcorner \circ \ulcorner w_2 \urcorner \circ \cdots \circ \ulcorner w_{i+1} \urcorner $$
Finally, since for each \(w \in \widetilde K\) we have \(\ulcorner w \urcorner = \pi_w\), it follows that:
$$ \ulcorner w_1 \urcorner \circ \cdots \circ \ulcorner w_n \urcorner = \pi_{w_1} \circ \cdots \circ \pi_{w_n}.$$
The result follows by mathematical induction.
\end{proof}


We need to define a morphism between any two finitary orthomodular dynamic algebras to establish category of finitary orthomodular dynamic algebras. We will refer to this as an $\mathbb{FODA}$-morphism in the following definition.

\begin{definition}\label{q-morphism}
Let $\mathfrak{K}_1 = (K_1 , {\bigsqcup }_1 , {\odot }_1 , \thicksim_1, {-^{*_1}} )$ and $\mathfrak{K}_2 = (K_2 , {\bigsqcup }_2 , {\odot }_2 , \thicksim_2, {-^{*_2}} )$ be finitary orthomodular dynamic algebras. A $\mathbb{FODA}$–morphism $\phi :{\mathfrak{K}_1} \to {\mathfrak{K}_2}$ is a function $\phi :{K_1} \to {K_2}$ satisfying the following conditions for all finite $A\subseteq K_1$ and $k,l \in K_1$:

\begin{enumerate}
    \item $\phi(\widetilde K_1)\subseteq \widetilde K_2$, and $\phi|_{\widetilde K_1}:\widetilde K_1 \to \widetilde K_2$ is an ortho-lattice isomorphism.
    \item $\phi\bigl(\bigsqcup_1 A\bigr) =\bigsqcup_2\{\phi(a)\mid a\in A\}$
    \item $\phi(k\odot_1 l) =\phi(k)\odot_2\phi(l)$
    \item $\phi\bigl(\sim_1 k \bigr) =\sim_2\bigl(\phi(k)\bigr)$
    \item $\phi\bigl(k^{*_1}\bigr) =\bigl(\phi(k)\bigr)^{*_2}$
    \item $\phi(e_1)=e_2$, where $e_1$ and $e_2$ are the unit elements of $K_1$ and $K_2$, respectively.
\end{enumerate}

The category of finitary orthomodular dynamic algebras and $\mathbb{FODA}$-morphisms is denoted by $\mathbb{FODA}$.
\end{definition}

\section{Main Results}

\subsection{The Functor from Orthomodular Lattices to Finitary Orthomodular Dynamic Algebras}

This subsection presents the construction of the functor $\Gamma: \mathbb{OML} \to \mathbb{FODA}$.

\subsubsection{Mapping of Objects}
For an arbitrary orthomodular lattice ${\mathcal{M}} = \left( {{M},{ \le },\mathop  - \nolimits^{{ \bot }} } \right)$, the components of $\Gamma(\mathcal{M})$ are defined as follows:

\begin{enumerate}
    \item $S_\mathcal{M}$: The smallest set comprising all Sasaki projections on $M$ (i.e., $\{\pi_m : M\to M \mid m\in M\}$) and closed under function composition ($\circ$).
    \item $\mathscr{P}_{\text{fin}}(S_\mathcal{M})$: The set of all finite subsets of $S_\mathcal{M}$.
    \item $\odot$: A binary operation on $\mathscr{P}_{\text{fin}}(S_\mathcal{M})$ defined by $A\odot B \;=\;\{\,a\circ b\mid a\in A,\;b\in B\}$ for every $A, B \in \mathscr{P}_{\text{fin}}(S_\mathcal{M})$.
    \item $\sim$: A unary operation on $\mathscr{P}_{\text{fin}}(S_\mathcal{M})$ defined by $\sim A =\left\{\pi_{\bigl(\,\bigvee\{\,a(1) \, \mid \, a\in A\}\bigr)^\perp} \right\}$.

    \item ${ - ^*}$: A unary operation on $\mathscr{P}_{\text{fin}}(S_\mathcal{M})$ defined by $A^* = \{\,a^*\mid a\in A\} $ for every $A \in \mathscr{P}_{\text{fin}}(S_\mathcal{M})$.
\end{enumerate}

Thus, $\Gamma (\mathcal {M})$ is set as $\bigl(\mathscr{P}_{\text{fin}}(S_\mathcal{M}) ,\bigcup,\odot,\sim, { - ^*}\bigr)$. We will demonstrate that for any orthomodular lattice $\mathcal{M}$, the structure $\Gamma (\mathcal M)$ is a finitary orthomodular dynamic algebra, satisfying all conditions of Definition~\ref{finitary-goda}. Preceding this demonstration, several preliminary lemmas essential for subsequent proofs will be presented.


\begin{lemma}
\label{lem:multiplication-properties}
Let ${\mathcal{M}} = \left( {{M},{ \le },\mathop  - \nolimits^{{ \bot }} } \right)$ be an orthomodular lattice and let 
$$\Gamma (\mathcal {M}) =\bigl(\mathscr{P}_{\text{fin}}(S_\mathcal{M}) ,\bigcup,\odot,\sim\bigr).$$
The operation $\odot$ on $\mathscr{P}_{\text{fin}}(S_\mathcal{M})$ exhibits the following properties:
\begin{enumerate}
    \item Associativity: For all $A, B, C \in \mathscr{P}_{\text{fin}}(S_\mathcal{M})$,
    \[
    (A\odot B)\odot C \;=\; A\odot(B\odot C)
    \]
    \item Distributivity over Finite Unions: For any finite index set $I$ and $A, B_i, A_i \in \mathscr{P}_{\text{fin}}(S_\mathcal{M})$, $\odot$ distributes over finite unions from both the left and the right:
    \[
    A \odot \Bigl(\bigcup_{i\in I}B_i\Bigr) \;=\; \bigcup_{i\in I}\bigl(A\odot B_i\bigr)
    \]
    \[
    \Bigl(\bigcup_{i\in I}A_i\Bigr)\odot B \;=\; \bigcup_{i\in I}(A_i\odot B)
    \]
    \item Existence of an Identity Element: There exists a unique identity element $E \in \mathscr{P}_{\text{fin}}(S_\mathcal{M})$ such that for all $A \in \mathscr{P}_{\text{fin}}(S_\mathcal{M})$,
    \[
    E\odot A = A = A\odot E
    \]
\end{enumerate}
\end{lemma}

\begin{proof}
\begin{enumerate}
    \item Associativity:
    Let $z \in (A\odot B)\odot C$. By definition of $\odot$, $z = (a \circ b) \circ c$ for some $a \in A$, $b \in B$, and $c \in C$. Since function composition ($\circ$) is associative in $S_\mathcal{M}$, we have $(a \circ b) \circ c = a \circ (b \circ c)$. Thus, $z \in A \odot (B \odot C)$. The reverse inclusion is analogous. Therefore, $(A\odot B)\odot C = A\odot(B\odot C)$.
\medskip
    \item Distributivity over Finite Unions:
    Let $z \in A \odot \left(\bigcup\limits_{i \in I} {{B_i}} \right)$. By definition, there exist $a \in A$ and $b \in \bigcup\limits_{i \in I} {{B_i}}$ such that $z = a \circ b$. Since $b \in \bigcup\limits_{i \in I} {{B_i}}$, there must be some $j \in I$ such that $b \in B_j$. Consequently, $z = a \circ b \in A \odot B_j \subseteq \bigcup\limits_{i \in I}(A \odot B_i)$. This establishes $A \odot \left(\bigcup\limits_{i \in I} B_i\right) \subseteq \bigcup\limits_{i \in I}(A \odot B_i)$.

    Conversely, let $z \in \bigcup\limits_{i \in I}(A \odot B_i)$. Then, for some $j \in I$, $z \in A \odot B_j$. This means $z = a \circ b$ for some $a \in A$ and $b \in B_j$. As $B_j \subseteq \bigcup\limits_{i \in I} B_i$, it follows that $b \in \bigcup\limits_{i \in I} B_i$. Hence, $z = a \circ b \in A \odot \left(\bigcup\limits_{i \in I} B_i\right)$. This shows $\bigcup\limits_{i \in I}(A \odot B_i) \subseteq A \odot \left(\bigcup\limits_{i \in I} B_i\right)$.
    Combining both inclusions, we conclude $A \odot \left(\bigcup\limits_{i \in I} B_i\right) = \bigcup_{i \in I}(A \odot B_i)$. The right distributivity, $\left(\bigcup\limits_{i \in I} A_i\right)\odot B = \bigcup\limits_{i \in I}(A_i \odot B)$, follows by an entirely symmetric argument.
\medskip
    \item Existence of an Identity Element:
    Let $E = \{\pi_1\}$, where $\pi_1: M \to M$ is the Sasaki projection at $1 \in M$. Note that $\pi_1$ is the identity map on $M$. For any finite set $A \subseteq S_\mathcal{M}$ (i.e., $A \in \mathscr{P}_{\text{fin}}(S_\mathcal{M})$), we have:
    \[
    E \odot A = \{\pi_1 \circ a \mid a \in A\} = \{a \mid a \in A\} = A
    \]
    Similarly,
    \[
    A \odot E = \{a \circ \pi_1 \mid a \in A\} = \{a \mid a \in A\} = A
    \]
    Thus, $E = \{\pi_1\}$ serves as the identity element for the $\odot$ operation in $\mathscr{P}_{\text{fin}}(S_\mathcal{M})$.
\end{enumerate}
\end{proof}


\begin{lemma}\label{lemma4}
Let ${\mathcal{M}} = \left( {{M},{ \le },\mathop  - \nolimits^{{ \bot }} } \right)$ be an orthomodular lattice and let $ A\in \mathscr{P}_{\text{fin}}(S_\mathcal{M})$.  Then $ \sim\sim A = \bigl\{\,\pi_{\bigl(\bigvee_{a\in A}a(1)\bigr)}\bigr\}$.
\end{lemma}

\begin{proof}
We know that $
\sim A
\;=\;
\bigl\{\pi_{\bigl(\bigvee_{a\in A}a(1)\bigr)^\perp}\bigr\}$. We must now compute \(\sim\sim A = \sim(\sim A)\).  By applying the same definition to the singleton
\(\{\,\pi_{(\bigvee_{a\in A}a(1))^\perp}\}\), we obtain
\[
\sim\sim A
=\;
\sim\bigl\{\;\pi_{(\bigvee_{a\in A}a(1))^\perp}\bigr\}
=\;
\Bigl\{\,\pi_{\bigl(\bigl(\pi_{(\bigvee_{a\in A}a(1))^\perp}\bigr)(1)\bigr)^\perp}\Bigr\}.
\]
Hence it suffices to show $\bigl(\pi_{(\bigvee_{a\in A}a(1))^\perp}\bigr)(1)
\;=\;
\bigl(\bigvee_{a\in A}a(1)\bigr)^\perp$, so that double‐negating returns \(\bigvee_{a\in A}a(1)\).  Recall that {\(\pi_m(1)=m\) for any \(m\in M\).} Applying this fact with $m \;=\; \bigl(\bigvee_{a\in A}a(1)\bigr)^\perp$,
we conclude
\[
\bigl(\pi_{(\bigvee_{a\in A}a(1))^\perp}\bigr)(1)
=\bigl(\bigvee_{a\in A}a(1)\bigr)^\perp.
\]
Therefore $\sim\sim A
=\Bigl\{\,\pi_{\bigl((\bigvee_{a\in A}a(1))^\perp\bigr)^\perp}\Bigr\}
=\Bigl\{\,\pi_{\bigvee_{a\in A}a(1)}\Bigr\}$, as required.
\end{proof}

The following proposition demonstrates that we can construct a structure isomorphic to any given orthomodular lattice, starting from that lattice itself.


\begin{proposition}\label{prop:chi-iso-finitary}
Let ${\mathcal{M}} = \left( {{M},{ \le },\mathop - \nolimits^{{ \bot }} } \right)$ be an orthomodular lattice. Let $ \widetilde {\mathscr{P}_{\text{fin}}(S_\mathcal{M})} =\{\sim W \mid W \in \mathscr{P}_{\text{fin}}(S_\mathcal{M})\}$.
Then the map $\delta : \mathcal{M} \to {\widetilde {\mathscr{P}_{\text{fin}}(S_\mathcal{M})}}$, defined by $\delta(m)=\{\pi_{m}\}$ for each $m \in M$, constitutes an ortho-lattice isomorphism.
\end{proposition}

\begin{proof}
To establish that $\delta : \mathcal{M} \to {\widetilde {\mathscr{P}_{\text{fin}}(S_\mathcal{M})}}$ is an ortho-lattice isomorphism, we must demonstrate that it is a bijective mapping that preserves both the join operation ($\vee$) and the orthocomplementation operation ($ - ^\perp$).

\medskip \noindent Injectivity:
Let $m,n \in M$. Assume $\delta(m)=\delta(n)$. By the definition of $\delta$, this implies:
\[ \{\pi_m\}=\{\pi_n\} \]
From this equality, it follows directly that $\pi_m = \pi_n$.
Given that $\pi_x(1) = x$ for any $x \in M$, we can deduce:
\[ m = \pi_m(1) = \pi_n(1) = n \]
Therefore, $\delta$ is injective.

\medskip \noindent Surjectivity:
We aim to show that for every $m \in M$, the singleton set $\{\pi_m\}$ is an element of $\widetilde {\mathscr{P}_{\text{fin}}(S_\mathcal{M})}$. This will demonstrate that the image of $\delta$ covers the entire codomain $\widetilde {\mathscr{P}_{\text{fin}}(S_\mathcal{M})}$.

Consider an arbitrary element $m \in M$. Let $A=\{\pi_{m^\perp}\} \subseteq S_\mathcal{M}$ be a finite set.
By definition of the operation $\sim$, we have:
\[ \sim A = \sim\{\pi_{m^\perp}\} = \left\{\,\pi_{\left(\pi_{m^\perp}(1)\right)^\perp}\right\} \]
We know that $\pi_{m^\perp}(1) = m^\perp \wedge (m \vee 1) = m^\perp \wedge 1 = m^\perp$.
Substituting this into the expression for $\sim A$:
\[ \sim\{\pi_{m^\perp}\} = \{\,\pi_{(m^\perp)^\perp}\} = \{\,\pi_m\} \]
Thus, $\{\pi_m\}=\sim\{\pi_{m^\perp}\}$, which implies that $\{\pi_m\} \in \widetilde {\mathscr{P}_{\text{fin}}(S_\mathcal{M})}$. Since $\delta(m)=\{\pi_m\}$, it follows that the range of $\delta$ covers all of $\widetilde {\mathscr{P}_{\text{fin}}(S_\mathcal{M})}$. Hence, $\delta$ is surjective.

\medskip\noindent Preservation of Joins:
Let $m,n \in M$. We need to show that $\delta(m \vee n) = \delta(m) \vee \delta(n)$.
In $\widetilde {\mathscr{P}_{\text{fin}}(S_\mathcal{M})}$, the join of $\delta(m)=\{\pi_m\}$ and $\delta(n)=\{\pi_n\}$ is defined as:
\[ \delta(m)\vee\delta(n) = \sim\sim\left(\{\pi_m\}\cup\{\pi_n\}\right) = \sim\sim\{\pi_m,\pi_n\} \]
By Lemma~\ref{lemma4}, we have:
\[ \sim\sim\{\,\pi_m,\pi_n\} = \{\,\pi_{\left(\,\pi_m(1)\vee \pi_n(1)\right)}\} \]
Since $\pi_x(1) = x$, this further simplifies to:
\[ \{\,\pi_{\left(\,\pi_m(1)\vee \pi_n(1)\right)}\} = \{\,\pi_{m\vee n}\} \]
Finally, by the definition of $\delta$:
\[ \{\,\pi_{m\vee n}\} = \delta(m\vee n) \]
Therefore, $\delta(m \vee n) = \delta(m) \vee \delta(n)$, demonstrating that $\delta$ preserves the join operation.

\medskip\noindent Preservation of Orthocomplement:
Let $m \in M$. We need to show that $\delta(m^\perp) = \sim\delta(m)$.
By the definition of $\delta$:
\[ \delta(m^\perp) = \{\,\pi_{m^\perp}\,\} \]
Now, consider $\sim\delta(m)$:
\[ \sim\delta(m) = \sim\{\pi_m\} \]
Using the definition of the $\sim$ operation on a singleton set, we have:
\[ \sim\{\pi_m\} = \left\{\,\pi_{\left(\pi_m(1)\right)^{\perp}}\right\} \]
Given that $\pi_m(1) = m$:
\[ \left\{\,\pi_{\left(\pi_m(1)\right)^{\perp}}\right\} = \{\,\pi_{m^{\perp}}\,\} \]
Combining these steps, we obtain:
\[ \delta(m^\perp) = \{\,\pi_{m^\perp}\,\} = \sim\{\,\pi_{m}\,\} = \sim\delta(m) \]
Thus, $\delta$ preserves the orthocomplement operation.

\medskip
Since $\delta$ has been shown to be bijective and to preserve both the join operation and the orthocomplementation, it is an ortho-lattice isomorphism.
\end{proof}

In the following theorem, we will show that ${\mathscr{P}_{\text{fin}}(S_\mathcal{M})}$ is the minimal set containing \(\widetilde {\mathscr{P}_{\text{fin}}(S_\mathcal{M})}\) which is closed under multiplication, join, and involution operations.

\begin{theorem}
The set ${\mathscr{P}_{\text{fin}}(S_\mathcal{M})}$ is the smallest set containing the singleton projections $\{\{\pi_m\}\mid m\in M\}$ and closed under the operations of product ($\odot$), finite join ($\bigsqcup$), involution (${ - ^*}$), and orthocomplementation ($\sim$).
\end{theorem}

\begin{proof}
By definition, ${\mathscr{P}_{\text{fin}}(S_\mathcal{M})}$ comprises all finite subsets of $S_\mathcal{M}$. We will demonstrate its closure under the specified operations and its minimality.

\begin{enumerate}[label=\arabic*.]
    \item Closure under Product: Let $A,B \in {\mathscr{P}_{\text{fin}}(S_\mathcal{M})}$. The product operation is defined as:
    \[ A\odot B = \{\,a\circ b\mid a\in A,\;b\in B\} \]
    Since $A$ and $B$ are finite sets, their Cartesian product $A \times B$ is also finite. Consequently, the set of compositions $\{\,a\circ b\mid (a,b)\in A \times B\}$ is finite. As this resulting set consists of elements from $S_\mathcal{M}$, it remains a finite subset of $S_\mathcal{M}$, thus $A \odot B \in {\mathscr{P}_{\text{fin}}(S_\mathcal{M})}$.

    \item Closure under Involution: Let $A \in {\mathscr{P}_{\text{fin}}(S_\mathcal{M})}$. The involution operation is defined as:
    \[ A^* = \{\,a^*\mid a\in A\} \]
    Since $A$ is finite, the set $A^*$ obtained by applying the involution to each element of $A$ is also finite. Therefore, $A^* \in {\mathscr{P}_{\text{fin}}(S_\mathcal{M})}$.

    \item Closure under Orthocomplementation: Let $A \in {\mathscr{P}_{\text{fin}}(S_\mathcal{M})}$. The orthocomplementation operation is defined as:
    \[ \sim A = \left\{\,\pi_{\left(\bigvee\{\,a(1)\mid a\in A\}\right)^\perp}\right\} \]
    This definition indicates that $\sim A$ is always a singleton set containing a single projection $\pi_x$ for some $x \in M$. As a singleton, it is inherently a finite subset of $S_\mathcal{M}$. Hence, $\sim A \in {\mathscr{P}_{\text{fin}}(S_\mathcal{M})}$.

    \item Closure under Finite Joins: Let $\{A_1,\dots,A_n\}$ be a finite collection of sets such that each $A_i \in {\mathscr{P}_{\text{fin}}(S_\mathcal{M})}$. The finite join operation is defined as:
    \[ \bigsqcup\{A_i\} = A_1\;\cup\;\cdots\;\cup\;A_n \]
    Since each $A_i$ is a finite set, their finite union is also a finite set. Therefore, $\bigsqcup\{A_i\} \in {\mathscr{P}_{\text{fin}}(S_\mathcal{M})}$.
\end{enumerate}

Finally, consider any other subset $\mathscr{P}'(S_\mathcal{M})$ of $S_\mathcal{M}$ that contains the set of singleton projections $\{\{\pi_m\}\mid m\in M\}$ and is closed under $\odot,\bigsqcup,^*,\text{and }\sim$. Because $\mathscr{P}'(S_\mathcal{M})$ must contain all initial singleton projections and is closed under the aforementioned operations, it must necessarily generate and thus encompass every finite subset of $S_\mathcal{M}$. In other words, $\mathscr{P}_{\text{fin}}(S_\mathcal{M}) \subseteq \mathscr{P}'(S_\mathcal{M})$. Consequently, ${\mathscr{P}_{\text{fin}}(S_\mathcal{M})}$ is the minimal such closed set.
\end{proof}

Within a finitary orthomodular dynamic algebra built from the orthomodular lattice $Q$, we can demonstrate that the equality of two sets in $\left\langle \widetilde {\mathscr{P}_{\text{fin}}(S_\mathcal{M})} \right\rangle$ (i.e., $A=B$) is equivalent to $A\equiv B$.

\begin{proposition}\label{prop:completeness}
Let ${\mathcal{M}} = \left( {{M},{ \le },\mathop  - \nolimits^{{ \bot }} } \right)$ be an orthomodular lattice and let 
$$\Gamma (\mathcal {M}) =\bigl(\mathscr{P}_{\text{fin}}(S_\mathcal{M}) ,\bigcup,\odot,\sim, { - ^*}\bigr).$$
For any 
$A,B\in \left\langle \widetilde {\mathscr{P}_{\text{fin}}(S_\mathcal{M})} \right\rangle$, the following conditions are equivalent:
\begin{enumerate}
  \item $A=B$.
  \item $A\equiv B$, which means that $\ulcorner A\urcorner(P)=\ulcorner B\urcorner(P)$ for every
    $P\in \widetilde {\mathscr{P}_{\text{fin}}(S_\mathcal{M})}$.
\end{enumerate}
\end{proposition}

\begin{proof}
The implication $(1)\Rightarrow(2)$ is a direct consequence of the definition of equivalence.

For the implication $(2)\Rightarrow(1)$, assume $A\equiv B$. Given that $A,B\in \left\langle \widetilde {\mathscr{P}_{\text{fin}}(S_\mathcal{M})} \right\rangle$, by definition, there exist positive integers $m,n \in \mathbb{N}^+$ and elements $q_1,\dots,q_m \in M$ and $r_1,\dots,r_n \in M$ such that $A=\{\;\pi_{q_1}\circ\cdots\circ \pi_{q_m}\;\}$ and $B=\{\;\pi_{r_1}\circ\cdots\circ \pi_{r_n}\;\}$.

According to Proposition \ref{prop:chi-iso-finitary}, for each $p\in M$, the singleton $\{\pi_p\}$ is an element of $\widetilde {\mathscr{P}_{\text{fin}}(S_\mathcal{M})}$. Consequently, the hypothesis $\ulcorner A\urcorner(P)=\ulcorner B\urcorner(P)$ specifically applies when $P=\{\pi_p\}$, yielding:
$$\sim\sim\bigl(A\cdot\{\pi_p\}\bigr) \;=\; \sim\sim\bigl(B\cdot\{\pi_p\}\bigr).$$
Substituting the definitions of $A$ and $B$, we have:
$$A \cdot \{\pi_p\} = \left\{\,\pi_{q_1} \circ \cdots \circ \pi_{q_m} \circ \pi_p\,\right\}$$
and
$$B \cdot \{\pi_p\} = \left\{\,\pi_{r_1} \circ \cdots \circ \pi_{r_n} \circ \pi_p\,\right\}.$$
Thus, the equality $\sim\sim\bigl( A \cdot \{\pi_p\} \bigr) = \sim\sim\bigl( B \cdot \{\pi_p\} \bigr)$ is equivalent to:
$$\sim\sim\left\{\,\pi_{q_1} \circ \cdots \circ \pi_{q_m} \circ \pi_p\,\right\} = \sim\sim\left\{\,\pi_{r_1} \circ \cdots \circ \pi_{r_n} \circ \pi_p\,\right\}.$$
By Lemma~\ref{lemma4}, it is known that:
$$\sim\sim\left\{ \pi_{x_1} \circ \cdots \circ \pi_{x_k} \right\} = \left\{ \pi_{(\pi_{x_1} \circ \cdots \circ \pi_{x_k})(1)} \right\}.$$
Applying this lemma to our current expression, and noting that $\pi_x(1) = x$ for any $x \in M$, the equality of the double orthocomplemented singleton sets reduces to:
$$(\pi_{q_1} \circ \cdots \circ \pi_{q_m})(p) = (\pi_{r_1} \circ \cdots \circ \pi_{r_n})(p)$$
for all $p \in M$. This implies that the compositions of Sasaki projections corresponding to $A$ and $B$, when applied to any element $p \in M$, yield identical results. Since this equality holds for every $p\in M$, it follows from \textbf{FODA 5} that $\pi_{q_1}\circ\cdots\circ \pi_{q_m} = \pi_{r_1}\circ\cdots\circ \pi_{r_n}$. This directly implies that $\{\,\pi_{q_1}\circ\cdots\circ \pi_{q_m}\} =\{\,\pi_{r_1}\circ\cdots\circ \pi_{r_n}\}$, and therefore, $A=B$.
\end{proof}



Furthermore, in the finitary orthomodular dynamic algebra constructed from the orthomodular lattice $Q$, it's proven that $\ulcorner A\urcorner(B)$ is precisely the Sasaki-projection of $B$ onto $A$.

\begin{proposition}\label{prop:fA}
Let ${\mathcal{M}} = \left( {{M},{ \le },\mathop  - \nolimits^{{ \bot }} } \right)$ be an orthomodular lattice and let 
$$\Gamma (\mathcal {M}) =\bigl(\mathscr{P}_{\text{fin}}(S_\mathcal{M}) ,\bigcup,\odot,\sim, { - ^*}\bigr).$$
Then for any \(A,B\in \widetilde {\mathscr{P}_{\text{fin}}(S_\mathcal{M})}\), it holds that $\ulcorner A\urcorner(B)\;=\;\pi_{A}(B)$.
\end{proposition}

\begin{proof}
By Proposition \ref{prop:chi-iso-finitary}, each \(A\in \widetilde {\mathscr{P}_{\text{fin}}(S_\mathcal{M})}\) and \(B\in \widetilde {\mathscr{P}_{\text{fin}}(S_\mathcal{M})}\) is a singleton of the form \(A=\delta(p)=\{\pi_p\}\) and \(B=\delta(q)=\{\pi_q\}\) for unique \(p,q\in M\). Therefore, we can compute $\ulcorner A\urcorner(B)$ as follows:
\begin{align*}
\ulcorner A\urcorner(B)
&=\sim\sim\bigl(A\cdot B\bigr) \\
&=\sim\sim\bigl(\{\pi_p\}\cdot\{\pi_q\}\bigr) \\
&=\sim\sim\{\pi_p\circ \pi_q\}\\
&=\{\,\pi_{\,\pi_p\circ \pi_q}(1)\}\\
&=\{\,\pi_p\bigl(\pi_q(1)\bigr)\}\\
&=\{\,\pi_p(q)\,\}.
\end{align*}
On the other hand, by definition of the Sasaki projection on \(\widetilde {\mathscr{P}_{\text{fin}}(S_\mathcal{M})}\), we have:
\begin{align*}
\pi_A(B)
&=A\wedge\bigl(\sim A\vee B\bigr) \\
&=\{\pi_p\}\wedge\bigl(\sim\{\pi_p\}\vee\{\pi_q\}\bigr)\\
&=\{\,\pi_{\,p\wedge(p^{\perp}\vee q)}\}\\
&=\{\,\pi_p(q)\,\}.
\end{align*}
The last equality holds because $\pi_p\circ \pi_q = \pi_{\,p\wedge(p^{\perp}\vee q)}$ and $\pi_q(1)=q$.
Consequently, we conclude that $\ulcorner A\urcorner(B)=\pi_A(B)$, as required.
\end{proof}

We will now show that in the finitary orthomodular dynamic algebra built from the orthomodular lattice $Q$, $\ulcorner A\urcorner(B)$ is equal to $\ulcorner A\urcorner(\sim\sim B)$.

\begin{theorem}\label{prop:foda7!}
Let ${\mathcal{M}} = \left( {{M},{ \le },\mathop  - \nolimits^{{ \bot }} } \right)$ be an orthomodular lattice and let 
$$\Gamma (\mathcal {M}) =\bigl(\mathscr{P}_{\text{fin}}(S_\mathcal{M}) ,\bigcup,\odot,\sim, { - ^*}\bigr).$$
Then, for all $A,B\in \mathscr{P}_{\text{fin}}(S_\mathcal{M})$, the following equality holds:
\[ \ulcorner A\urcorner(B) = \ulcorner A\urcorner(\sim\sim B). \]
\end{theorem}

\begin{proof}
By definition, $\ulcorner A\urcorner(Y) = \sim\sim\bigl(A\odot Y\bigr)$ for all $A,Y\in \mathscr{P}_{\text{fin}}(S_\mathcal{M})$.
From the definition of $\sim$, we know that $\sim B = \bigl\{\,\pi_{\,(\bigvee_{b\in B}b(1))^\perp}\bigr\}$. Consequently, by Lemma~\ref{lemma4}, it follows that $\sim\sim B = \bigl\{\,\pi_{\bigvee_{b\in B}b(1)}\bigr\}$.

Now, consider the expression $\ulcorner A\urcorner(\sim\sim B)$:
\[ A\odot(\sim\sim B) = \bigl\{\,a\circ \pi_{\bigvee_{b\in B}b(1)}\bigm|a\in A\bigr\} \]
Therefore,
\[ \ulcorner A\urcorner(\sim\sim B) = \sim\sim \bigl\{\,a\circ \pi_{\bigvee_{b\in B}b(1)}\bigm|a\in A\bigr\} = \bigl\{\,\pi_{\bigvee_{a\in A}(a\circ \pi_{\bigvee_{b\in B}b(1)})(1)}\bigr\}. \]
Moreover, for every \(a\in A\), the term $(a\circ \pi_{\bigvee_{b\in B}b(1)})(1)$ expands to:
\[ (a\circ \pi_{\bigvee_{b\in B}b(1)})(1) = a\bigl(\pi_{\bigvee_{b\in B}b(1)}(1)\bigr) = a\bigl(\bigvee_{b\in B}b(1)\bigr) = \bigvee_{b\in B}a\bigl(b(1)\bigr). \]

On the other hand, let's evaluate $\ulcorner A\urcorner(B)$:
\[ \ulcorner A\urcorner(B) = \sim\sim\{\,a\circ b\mid a\in A,\;b\in B\} = \bigl\{\,\pi_{\bigvee_{a\in A,\,b\in B}a(b(1))}\bigr\}. \]
Since $\bigvee_{a\in A,\,b\in B}a(b(1)) = \bigvee_{a\in A}\!\bigl(\bigvee_{b\in B}a(b(1))\bigr)$, the arguments of $\pi$ in the expressions for $\ulcorner A\urcorner(B)$ and $\ulcorner A\urcorner(\sim\sim B)$ are identical.

Thus, we conclude that $\ulcorner A\urcorner(B)=\ulcorner A\urcorner(\sim\sim B)$.
\end{proof}

Finally, we will prove that $\mathbf F(\mathcal L)$ is a finitary orthomodular dynamic algebra.

\begin{theorem}
Let ${\mathcal{M}} = \left( {{M},{ \le },\mathop - \nolimits^{{ \bot }} } \right)$ be an orthomodular lattice. The structure $\Gamma (\mathcal {M}) =\bigl(\mathscr{P}_{\text{fin}}(S_\mathcal{M}) ,\bigcup,\odot,\sim, { - ^*}\bigr)$ constitutes a finitary orthomodular dynamic algebra.
\end{theorem}

\begin{proof}
To establish that $\Gamma (\mathcal {M})$ is a finitary orthomodular dynamic algebra, we verify that axioms \textbf{FODA 1} through \textbf{FODA 7} from Definition \ref{finitary-goda} are satisfied.

\begin{enumerate}[label=(\textbf{FODA \arabic*}),leftmargin=1.5cm]
 \item By Lemma \ref{lem:multiplication-properties}, the triple $\bigl(\mathscr{P}_{\text{fin}}(S_\mathcal{M}),\;\subseteq,\;\odot\bigr)$ forms a unital involutive $m$-semilattice with unit $\{\pi_1\}$. The operation $\bigsqcup$ is precisely the finite-join operation, equivalent to set union ($\cup$). The involution $*$ is defined by reversing compositions and satisfies the properties $(A^*)^*=A$, $(A\vee B)^*=A^*\vee B^*$, and $(A\odot B)^*=B^*\odot A^*$.

 \item Proposition \ref{prop:chi-iso-finitary} demonstrates that $(\mathcal P_{\mathfrak Q},\;\preceq,\;\sim)$ is an orthomodular lattice. It is also evident that if $A \in \mathcal{P}$, then $A^* \in \mathcal{P}$.

 \item Minimality: Any subset $A\subseteq \mathscr{P}_{\text{fin}}(S_\mathcal{M})$ that contains $\mathcal P$ and is closed under the operations $\odot,\,*$, and $\bigsqcup$ must necessarily encompass all finite subsets of $S_\mathcal{M}$. This implies $A = \mathscr{P}_{\text{fin}}(S_\mathcal{M})$, as established by Theorem \ref{lemma2} and its accompanying closure arguments.

 \item Uniqueness of Join-Decomposition: The uniqueness of join-decomposition was rigorously proven in Theorem \ref{lemma2}.

\item Completeness: For any $A,B\in \left\langle \widetilde {\mathscr{P}_{\text{fin}}(S_\mathcal{M})} \right\rangle$, the equivalence $A=B$ holds if and only if $\ulcorner A\urcorner(P)=\ulcorner B\urcorner(P)$ for all $P\in \widetilde {\mathscr{P}_{\text{fin}}(S_\mathcal{M})}$. This property is detailed in Proposition \ref{prop:completeness}.

 \item Sasaki Projection: For any $P,Q \in \widetilde {\mathscr{P}_{\text{fin}}(S_\mathcal{M})}$, it holds that $\ulcorner P\urcorner(Q)=\pi_P(Q)$, as shown in Proposition \ref{prop:fA}.

 \item Composition Law: For all $A,B\in \mathscr{P}_{\text{fin}}(S_\mathcal{M})$, the composition law $\ulcorner A\urcorner(B)=\ulcorner A\urcorner(\sim\sim B)$ is satisfied, as demonstrated in Proposition \ref{composition}.
\end{enumerate}

Having verified that axioms \textbf{FODA 1} through \textbf{FODA 7} hold, we conclude that $\Gamma (\mathcal {M})$ is indeed a finitary orthomodular dynamic algebra.
\end{proof}


\subsubsection{Mapping of Arrows}
We formally define how the functor $\Gamma$ maps ortho-lattice isomorphisms between orthomodular lattices. Let $\mathcal{M}_1=(M_{1},\le_{1}, -^{\perp_{1}})$ and $\mathcal{M}_2=(M_{2}, \le_{2}, -^{\perp_{2}})$ be orthomodular lattices. The functor $\Gamma$ maps these lattices to $\Gamma (\mathcal{M}_1) =\bigl(\mathscr{P}_{\text{fin}}(S_{\mathcal{M}_1}) ,\bigcup_1,\odot_1,\sim_1, -^{*_1} \bigr)$ and $\Gamma (\mathcal{M}_2) =\bigl(\mathscr{P}_{\text{fin}}(S_{\mathcal{M}_2}) ,\bigcup_2,\odot_2,\sim_2, -^{*_2} \bigr)$ respectively.

Given an ortho-lattice isomorphism $k : \mathcal{M}_1 \to \mathcal{M}_2$, we define the mapping $\Gamma(k) : \Gamma (\mathcal{M}_1) \to \Gamma (\mathcal{M}_2)$ as follows:

$$A \mapsto \{k \circ a \circ k^{-1} \mid a \in A\}$$

For any $m \in M_1$, we can demonstrate the equivalence $\pi_{\,k(m)}=k\circ \pi_m\circ k^{-1}$ through the following derivation, which holds for every $n \in M_1$:

\[
\begin{aligned}
{\pi _{k\left( m \right)}}\left( n \right) &= k\left( m \right) \wedge \left( {{{\left( {k\left( m \right)} \right)}^ \bot } \vee n} \right)\\
&= k\left( {m \wedge \left( {{m^ \bot } \vee {k^{ - 1}}\left( n \right)} \right)} \right)\\
&= k\left( {{\pi _m}\left( {{k^{ - 1}}\left( n \right)} \right)} \right)\\
&=\left( {k \circ {\pi _m} \circ {k^{ - 1}}} \right)\left( n \right).
\end{aligned}
\]


The next theorem defines an $\mathbb{FODA}$-morphism specifically for the finitary orthomodular dynamic algebras built from orthomodular lattices.

\begin{theorem}
Let $\mathcal M_{1}=(M_{1},\le_{1}, -^{\perp_{1}})$ and $\mathcal M_{2}=(M_{2}, \le_{2}, -^{\perp_{2}})$ be orthomodular lattices. If $k: \mathcal M_{1} \to \mathcal M_{2}$ is an ortho-lattice isomorphism, then the mapping $\Gamma(k) : \Gamma (\mathcal{M}_1) \to \Gamma (\mathcal{M}_2)$, defined by $A \mapsto \{k \circ a \circ k^{-1} \mid a \in A\}$, is an $\mathbb{FODA}$-morphism.
\end{theorem}

\begin{proof}
To demonstrate that $\Gamma(k)$ is an $\mathbb{FODA}$-morphism, we must establish two conditions: first, that it constitutes an ortho-lattice isomorphism from $\left( {\widetilde {\mathscr{P}_{\text{fin}}(S_{\mathcal{M}_1})},{ \preceq _1},{ \sim _1}} \right)$ to $\left( {\widetilde {\mathscr{P}_{\text{fin}}(S_{\mathcal{M}_2})},{ \preceq _2},{ \sim _2}} \right)$, and second, that it preserves the finite join $\bigsqcup$, the multiplication $\odot$, and the involution $*$.

By Proposition \ref{prop:chi-iso-finitary}, we have $\widetilde {\mathscr{P}_{\text{fin}}(S_{\mathcal{M}_1})}=\bigl\{\{\,\pi_m\}\mid m \in M_1\bigr\}$ and $\widetilde {\mathscr{P}_{\text{fin}}(S_{\mathcal{M}_2})}=\bigl\{\{\,\pi_n\}\mid n \in M_2\bigr\}$. Observe that $\Gamma \left( k \right) = {\delta _2} \circ k \circ \delta _1^{ - 1}$ on $\widetilde {\mathscr{P}_{\text{fin}}(S_{\mathcal{M}_1})}$, given that for every $\{\pi_p\} \in \widetilde {\mathscr{P}_{\text{fin}}(S_{\mathcal{M}_1})}$, the following holds:
\begin{align*}
\Gamma(k)\bigl(\{\pi_p\}\bigr)
& = \{\,k\circ \pi_p\circ k^{-1}\}\\
& = \{\,\pi_{k(p)}\}\\
& = \delta_2\bigl(k(p)\bigr)\\
& = (\delta_2\circ k\circ\delta_1^{-1})\bigl(\{\pi_p\}\bigr).
\end{align*}

Next, we verify the preservation properties of $\Gamma(k)$.

\noindent \textit{Preservation of Finite Join $\bigsqcup$}: Let $A,B \in {\mathscr{P}_{\text{fin}}(S_{\mathcal{M}_1})}$. Then:
\begin{align*}
\Gamma(k)(A\bigsqcup_1 B)
& = \Gamma(k)(A\cup B) \\[6pt]
& = \{\,k\circ c\circ k^{-1}\mid c\in A\cup B\} \\[6pt]
& = \{\,k\circ a\circ k^{-1}\mid a\in A\}
 \;\cup\;
 \{\,k\circ b\circ k^{-1}\mid b\in B\} \\[6pt]
& = \Gamma(k)(A)\;\cup\;\Gamma(k)(B) \\[6pt]
& = \Gamma(k)(A)\bigsqcup_2\Gamma(k)(B).
\end{align*}

\noindent \textit{Preservation of Multiplication $\odot$}: Let $A,B\in {\mathscr{P}_{\text{fin}}(S_{\mathcal{M}_1})}$. Then:
\begin{align*}
\Gamma(k)(A\odot_1 B)
& = \Gamma(k)\bigl(\{a\circ b\mid a\in A,b\in B\}\bigr)\\
& = \{\,k\circ(a\circ b)\circ k^{-1}\mid a\in A,b\in B\}.
\end{align*}
Since $k$ is a homomorphism of functions, we have:
\begin{align*}
k\circ(a\circ b)\circ k^{-1}
& = (k\circ a\circ k^{-1})\;\circ\;(k\circ b\circ k^{-1}).
\end{align*}
Thus, it follows that:
\begin{align*}
\Gamma(k)(A\odot_1 B)
& = \{(k\circ a\circ k^{-1})\circ(k\circ b\circ k^{-1})\mid a\in A,b\in B\}\\
& = \Gamma(k)(A)\odot_2\Gamma(k)(B).
\end{align*}

\noindent \textit{Preservation of Involution $-^*$}: Let $A\in {\mathscr{P}_{\text{fin}}(S_{\mathcal{M}_1})}$. By definition, $A^* = \{\,a^*\mid a\in A\}$. Consequently:
\begin{align*}
\Gamma(k)(A^*)
& = \{\,k\circ a^*\circ k^{-1}\mid a\in A\}.
\end{align*}
If we express $a$ as $\pi_{i_1}\circ\cdots\circ \pi_{i_n}$, then $a^*=\pi_{i_n}\circ\cdots\circ \pi_{i_1}$. It can be verified that:
\begin{align*}
k\circ a^*\circ k^{-1}
& = \bigl(k\circ \pi_{i_n}\circ k^{-1}\bigr)\circ\cdots\circ\bigl(k\circ \pi_{i_1}\circ k^{-1}\bigr)\\
& = \bigl(k\circ(\pi_{i_1}\circ\cdots\circ \pi_{i_n})\circ k^{-1}\bigr)^*\\
& = (k\circ a\circ k^{-1})^*.
\end{align*}
Therefore:
\begin{align*}
\Gamma(k)(A^*)
& = \{(k\circ a\circ k^{-1})^*\mid a\in A\}\\
& = \bigl\{\,(k\circ a\circ k^{-1})\mid a\in A\bigr\}^*\\
& = \Gamma(k)(A)^*.
\end{align*}

Since $\Gamma(k)$ is a bijection on $\widetilde {\mathscr{P}_{\text{fin}}(S_{\mathcal{M}_1})}$ and preserves $\bigsqcup$, $\odot$, and $-^*$, it is established that $\Gamma(k)$ is an $\mathbb{FODA}$-morphism between $\Gamma(\mathcal{M}_1)$ and $\Gamma(\mathcal{M}_2)$.
\end{proof}

\begin{theorem}\label{comp}
Let $\mathcal{M}_1$, $\mathcal{M}_2$, and $\mathcal{M}_3$ be orthomodular lattices. If $k: \mathcal{M}_1 \to \mathcal{M}_2$ and $l: \mathcal{M}_2 \to \mathcal{M}_3$ are $\mathbb{OML}$-morphisms, then the mapping $\Gamma$ preserves the composition of morphisms, i.e., $\Gamma(l \circ k) = \Gamma(l) \circ \Gamma(k)$.
\end{theorem}

\begin{proof}
Let $k: \mathcal{M}_1 \to \mathcal{M}_2$ and $l: \mathcal{M}_2 \to \mathcal{M}_3$ be arbitrary $\mathbb{OML}$-morphisms. Our objective is to demonstrate that $\Gamma(l \circ k) = \Gamma(l) \circ \Gamma(k)$.

For any $A_1 \in Q_1$, we apply the definition of $\Gamma$ to the composite morphism $l \circ k$:
\begin{align*}
\Gamma(l \circ k)(A_1) & = \{(l \circ k) \circ a \circ (l \circ k)^{-1} \mid a \in A_1\}.
\end{align*}
Next, we apply $\Gamma(l)$ to the result of $\Gamma(k)(A_1)$:
\begin{align*}
\Gamma(l)(\Gamma(k)(A_1)) & = \Gamma(l)(\{k \circ a \circ k^{-1} \mid a \in A_1\}) \\
& = \{l \circ (k \circ a \circ k^{-1}) \circ l^{-1} \mid a \in A_1\}.
\end{align*}
By the associativity of function composition, we can re-arrange the terms within the set:
\begin{align*}
l \circ (k \circ a \circ k^{-1}) \circ l^{-1} & = (l \circ k) \circ a \circ (k^{-1} \circ l^{-1}) \\
& = (l \circ k) \circ a \circ (l \circ k)^{-1}.
\end{align*}
Therefore, substituting this back into the expression for $\Gamma(l)(\Gamma(k)(A_1))$:
\begin{align*}
\Gamma(l)(\Gamma(k)(A_1)) & = \{(l \circ k) \circ a \circ (l \circ k)^{-1} \mid a \in A_1\}.
\end{align*}
Comparing this result with the definition of $\Gamma(l \circ k)(A_1)$, we conclude that:
\begin{align*}
\Gamma(l \circ k)(A_1) & = \Gamma(l)(\Gamma(k)(A_1)).
\end{align*}
Since this equality holds for all $A_1 \in Q_1$, it follows that $\Gamma(l \circ k) = \Gamma(l) \circ \Gamma(k)$.
\end{proof}

We also show that $\Gamma$ preserves identity morphisms.
\begin{theorem}\label{id}
Let $\mathcal{M}$ be an orthomodular lattice. The functor $\Gamma$ preserves the identity morphism on $\mathcal{M}$, that is, $\Gamma(\text{id}_{\mathcal{M}}) = \text{id}_{\Gamma(\mathcal{M})}$.
\end{theorem}

\begin{proof}
Let $\mathcal{M} = (M, \le, -^{\perp})$ be an orthomodular lattice. Let $\text{id}_{\mathcal{M}}$ denote the identity morphism on $\mathcal{M}$ within the category of orthomodular lattices, $\mathbb{OML}$. Our objective is to demonstrate that $\Gamma(\text{id}_{\mathcal{M}}) = \text{id}_{\Gamma(\mathcal{M})}$.

By definition of the functor $\Gamma$, for any object $A \in Q$ (where $Q$ is the image of $\mathcal{M}$ under the functor $\Gamma$):
\[ \Gamma(\text{id}_{\mathcal{M}})(A) = \{\text{id}_{\mathcal{M}} \circ a \circ \text{id}_{\mathcal{M}}^{-1} \mid a \in A\} \]
Since $\text{id}_{\mathcal{M}}$ is the identity morphism, for any $a \in A$, we have $\text{id}_{\mathcal{M}} \circ a = a$ and $a \circ \text{id}_{\mathcal{M}}^{-1} = a$. Thus, the expression simplifies to:
\[ \Gamma(\text{id}_{\mathcal{M}})(A) = \{a \mid a \in A\} = A \]
Since this holds for every $A \in Q$, it follows directly from the definition of the identity morphism on $\Gamma(\mathcal{M})$ that $\Gamma(\text{id}_{\mathcal{M}})$ acts as the identity on $\Gamma(\mathcal{M})$.
Therefore, we conclude that $\Gamma(\text{id}_{\mathcal{M}}) = \text{id}_{\Gamma(\mathcal{M})}$.
\end{proof}

This leads us to the conclusion that $\Gamma$ is a functor between $\mathbb{OML}$ and $\mathbb{FODA}$.

\begin{theorem}
The mapping $\Gamma$ constitutes a functor from the category of orthomodular lattices, $\mathbb{OML}$, to the category of finitary orthomodular dynamic algebras, $\mathbb{FODA}$.
\end{theorem}

\begin{proof}
The assertion follows directly from Theorem \ref{comp}, which establishes the preservation of composition, and Theorem \ref{id}, which establishes the preservation of identity morphisms.
\end{proof}


We now proceed to show that the functor $\Gamma$ maintains the structure defined by the operation \(\sim\).

\begin{theorem}
Let $\mathcal{M}_1$ and $\mathcal{M}_2$ be orthomodular lattices. If $k: \mathcal{M}_1 \to \mathcal{M}_2$ is an ortho-lattice isomorphism, then for every $A \in \Gamma(\mathcal{M}_1)$, the following equality holds:
\[ \Gamma(k)(\sim A) = \sim(\Gamma(k)(A)) \]
\end{theorem}

\begin{proof}
Recall the definition of $\sim A$ as $\sim A = \{\pi_{(\bigvee\{a(1) \mid a \in A\})^\perp}\}$.
Applying the functor $\Gamma(k)$ to $\sim A$, we get:
\[ \Gamma(k)(\sim A) = \Gamma(k)\left(\{\pi_{(\bigvee\{a(1) \mid a \in A\})^\perp}\}\right) \]
By the definition of $\Gamma(k)$, this expands to:
\[ \Gamma(k)(\sim A) = \{k \circ \pi_{(\bigvee\{a(1) \mid a \in A\})^\perp} \circ k^{-1}\} \]
It is a known property that for any projection $\pi_p$ and an ortho-lattice isomorphism $k$, $k \circ \pi_p \circ k^{-1} = \pi_{k(p)}$.
Applying this property, we can write:
\[ k \circ \pi_{(\bigvee\{a(1) \mid a \in A\})^\perp} \circ k^{-1} = \pi_{k((\bigvee\{a(1) \mid a \in A\})^\perp)} \]
Since $k$ is an ortho-lattice isomorphism, it preserves the orthocomplement and joins, so $k((X)^\perp) = (k(X))^\perp$ and $k(\bigvee S) = \bigvee k(S)$. Thus:
\[ \pi_{k((\bigvee\{a(1) \mid a \in A\})^\perp)} = \pi_{(\bigvee\{k(a(1)) \mid a \in A\})^\perp} \]
Substituting this back into our expression for $\Gamma(k)(\sim A)$:
\[ \Gamma(k)(\sim A) = \{\pi_{(\bigvee\{k(a(1)) \mid a \in A\})^\perp}\} \]
By the definition of $\sim$ applied to $\Gamma(k)(A)$, which is defined as $\{\,k \circ a \circ k^{-1} \mid a \in A\}$, we observe that the element $k(a(1))$ corresponds to the first component of the elements in $\Gamma(k)(A)$. Therefore:
\[ \{\pi_{(\bigvee\{k(a(1)) \mid a \in A\})^\perp}\} = \sim(\Gamma(k)(A)) \]
Hence, we conclude that $\Gamma(k)(\sim A) = \sim(\Gamma(k)(A))$.
\end{proof}



\begin{proposition}\label{bijection}
Let $\mathcal{M}_1$ and $\mathcal{M}_2$ be orthomodular lattices. If $k: \mathcal{M}_1 \to \mathcal{M}_2$ is an ortho-lattice isomorphism, then the mapping $\Gamma(k) : \Gamma(\mathcal{M}_1) \to \Gamma(\mathcal{M}_2)$ is bijective. Moreover, its inverse is $\Gamma(k^{-1}): \Gamma(\mathcal{M}_2) \to \Gamma(\mathcal{M}_1)$, where $\Gamma(k^{-1})(C)=\{\,k^{-1}\circ c\circ k\mid c\in C\}$ for every $C \in \Gamma(\mathcal{M}_2)$.
\end{proposition}

\begin{proof}
Since $k$ is an ortho-lattice isomorphism, its inverse $k^{-1}: \mathcal{M}_2 \to \mathcal{M}_1$ is also an ortho-lattice isomorphism. Consequently, $\Gamma(k^{-1})$ is well-defined on $\mathscr{P}_{\text{fin}}(S_{\mathcal{M}_2})$, where $S_{\mathcal{M}_2}$ denotes the set of states of $\mathcal{M}_2$.

To demonstrate that $\Gamma(k)$ is bijective, we will show that $\Gamma(k^{-1})$ acts as both its left and right inverse. That is, we will prove:
\begin{enumerate}
    \item $\Gamma(k^{-1})\circ \Gamma(k)=\mathrm{id}_{\Gamma(\mathcal{M}_1)}$
    \item $\Gamma(k)\circ \Gamma(k^{-1})=\mathrm{id}_{\Gamma(\mathcal{M}_2)}$
\end{enumerate}

For any $A \in \Gamma(\mathcal{M}_1)$, we compute the composition $\Gamma(k^{-1})\bigl(\Gamma(k)(A)\bigr)$:
\[
\begin{aligned}
\Gamma(k^{-1})\bigl(\Gamma(k)(A)\bigr)
&= \bigl\{\,k^{-1}\circ x\circ k \;\big|\; x\in \Gamma(k)(A)\bigr\} && (\text{Definition of } \Gamma(k^{-1}))\\[4pt]
&= \bigl\{\,k^{-1}\circ (k\circ a\circ k^{-1})\circ k \;\big|\; a\in A\bigr\} && (\text{Definition of } \Gamma(k))\\[4pt]
&= \bigl\{\,(k^{-1}\circ k)\circ a\circ (k^{-1}\circ k) \;\big|\; a\in A\bigr\} && (\text{Associativity of composition})\\[4pt]
&= \{\,\mathrm{id}_{\mathcal{M}_1}\circ a\circ \mathrm{id}_{\mathcal{M}_1} \;\big|\; a\in A\} && (\text{Definition of inverse})\\[4pt]
&= \{\,a \mid a\in A\} && (\text{Property of identity})\\[4pt]
&= A
\end{aligned}
\]
Thus, $\Gamma(k^{-1})\circ \Gamma(k)=\mathrm{id}_{\Gamma(\mathcal{M}_1)}$.

Similarly, for any $C \in \Gamma(\mathcal{M}_2)$, we compute the composition $\Gamma(k)\bigl(\Gamma(k^{-1})(C)\bigr)$:
\[
\begin{aligned}
\Gamma(k)\bigl(\Gamma(k^{-1})(C)\bigr)
&= \bigl\{\,k\circ y\circ k^{-1} \;\big|\; y\in \Gamma(k^{-1})(C)\bigr\} && (\text{Definition of } \Gamma(k))\\[4pt]
&= \bigl\{\,k\circ (k^{-1}\circ c\circ k)\circ k^{-1} \;\big|\; c\in C\bigr\} && (\text{Definition of } \Gamma(k^{-1}))\\[4pt]
&= \bigl\{\,(k\circ k^{-1})\circ c\circ (k\circ k^{-1}) \;\big|\; c\in C\bigr\} && (\text{Associativity of composition})\\[4pt]
&= \{\,\mathrm{id}_{\mathcal{M}_2}\circ c\circ \mathrm{id}_{\mathcal{M}_2} \;\big|\; c\in C\} && (\text{Definition of inverse})\\[4pt]
&= \{\,c \mid c\in C\} && (\text{Property of identity})\\[4pt]
&= C
\end{aligned}
\]
Thus, $\Gamma(k)\circ \Gamma(k^{-1})=\mathrm{id}_{\Gamma(\mathcal{M}_2)}$.

Since $\Gamma(k^{-1})$ serves as both a left and right inverse for $\Gamma(k)$, it follows that $\Gamma(k)$ is bijective and its inverse is precisely $\Gamma(k^{-1})$.
\end{proof}

\subsection{The Functor from  Finitary Orthomodular Dynamic Algebras to Orthomodular Lattices} 

In this subsection we define a functor $ \Psi : \mathbb{FODA} \to \mathbb{OML}$,
which assigns to each finitary orthomodular dynamic algebra an orthomodular lattice and to each $\mathbb{FODA}$-morphism an ortho-lattice isomorphism.

\subsubsection{Mapping of Objects} 
Let $\mathfrak{K} = (K , {\bigsqcup } , {\odot } , \thicksim, {-^*} )$ denote a finitary orthomodular dynamic algebra. The object mapping of $\Psi$ is defined such that $\Psi(\mathfrak{K})=( \widetilde K,\preceq ,(-)^\perp)$, where $\widetilde K = \{ \thicksim k \mid k \in K \}$. Pursuant to Definition \ref{deffoda}, it is established that $\mathfrak{K}$ constitutes an orthomodular lattice.

\subsubsection{Mapping of Arrows} 
 Let $\mathfrak{K}_1 = (K_1 , {\bigsqcup }_1 , {\odot }_1 , \thicksim_1, {-^{*_1}} )$ and $\mathfrak{K}_2 = (K_2 , {\bigsqcup }_2 , {\odot }_2 , \thicksim_2, {-^{*_2}} )$ be finitary orthomodular dynamic algebras. Given that $\phi: \mathfrak{K}_1 \to\mathfrak{K}_2$ is a $\mathbb{FODA}$-morphism, the arrow mapping of $\Psi$ is defined as $\Psi(\phi) : \Psi(\mathfrak{K}_1) \to \Psi(\mathfrak{K}_2)$ such that $\Psi(\phi)(k) = \phi(k)$ for all $k \in K_1$.

\begin{theorem}
Let $\mathfrak{K}_1 = (K_1 , {\bigsqcup }_1 , {\odot }_1 , \thicksim_1, {-^{*_1}} )$ and $\mathfrak{K}_2 = (K_2 , {\bigsqcup }_2 , {\odot }_2 , \thicksim_2, {-^{*_2}} )$ be finitary orthomodular dynamic algebras. If $\phi: \mathfrak{K}_1 \to\mathfrak{K}_2$ is a $\mathbb{FODA}$-morphism, then $\Psi(\phi)(k)$ is an $\mathbb{OML}$-morphism from $\Psi(\mathfrak{K}_1)$ to $\Psi(\mathfrak{K}_2)$.
\end{theorem}

\begin{proof}
This assertion is a direct consequence of the definition of $\mathbb{FODA}$-morphisms.
\end{proof}

\begin{theorem}\label{thm:G-is-functor}
The mapping $\Psi$ constitutes a functor from the category of finitary orthomodular dynamic algebras ($\mathbb{FODA}$) to the category of orthomodular lattices ($\mathbb{OML}$).
\end{theorem}

\begin{proof}
Let $\mathfrak{K} = (K , {\bigsqcup } , {\odot } , \thicksim, {-^*} )$ be a finitary orthomodular dynamic algebra.

\noindent \textit{Identity}: The identity $\mathbb{FODA}$-morphism $\mathrm{id}_{\mathfrak{K}}$ restricts to the identity on $\widetilde K$. Consequently, for all $v \in \widetilde K$, we have $\Psi(\mathrm{id}_\mathfrak{K})(v) = \mathrm{id}_\mathfrak{K}(v) = v$. Thus, $\Psi(\mathrm{id}_\mathfrak{K})=\mathrm{id}_{\widetilde K}$.

\noindent \textit{Composition}: Let $\phi: \mathfrak{K}_1 \to \mathfrak{K}_2$ and $\varphi:\mathfrak{K}_2 \to \mathfrak{K}_3$ be $\mathbb{FODA}$-morphisms. Then for every $v \in \widetilde K_1$, we have:
\begin{align*}
\Psi(\varphi\circ\phi)(v) &= (\varphi\circ\phi)(v) \\
&= \varphi\bigl(\phi(v)\bigr) \\
&= \Psi(\varphi)\bigl(\phi(v)\bigr) \\
&= \bigl(\Psi(\varphi)\circ \Psi(\phi)\bigr)(v)
\end{align*}
Therefore, $\Psi(\varphi\circ\phi)=\Psi(\varphi)\circ \Psi(\phi)$.
\end{proof}

\subsection{The Equivalence between $\mathbb{OML}$ and $\mathbb{FODA}$}
This subsection is devoted to establishing a categorical equivalence between the category of orthomodular lattices ($\mathbb{OML}$) and the category of finitary orthomodular dynamic algebras ($\mathbb{FODA}$).

\subsubsection{The Natural Isomorphism $\mu : 1_{\mathbb{OML}} \to \Psi\circ \Gamma$}
We begin by constructing, for each orthomodular lattice ${\mathcal{M}} = \left( {{M},{ \le },\mathop - \nolimits^{{ \bot }} } \right)$, a corresponding finitary orthomodular dynamic algebra, denoted $\Gamma (\mathcal {M}) =\bigl(\mathscr{P}_{\text{fin}}(S_\mathcal{M}) ,\bigcup,\odot,\sim, { - ^*}\bigr)$. Let $\widetilde {\mathscr{P}_{\text{fin}}(S_\mathcal{M})}$ be defined as the set $\{\sim W \mid W \in \mathscr{P}_{\text{fin}}(S_\mathcal{M})\}$. As shown in Theorem~\ref{prop:chi-iso-finitary}, the map $\delta : \mathcal{M} \to {\widetilde {\mathscr{P}_{\text{fin}}(S_\mathcal{M})}}$, where $\delta(m)=\{\pi_{m}\}$ for each $m \in M$, is an ortho-lattice isomorphism. Through the definition of the functor $\Psi$, we observe that
\begin{center}
$\left( {\Psi \circ \Gamma } \right)\left( M \right) = \Psi \left( {\Gamma \left( M \right)} \right) = \Psi \left( \mathscr{P}_{\text{fin}}(S_\mathcal{M}) \right) = \widetilde {\mathscr{P}_{\text{fin}}(S_\mathcal{M})}$.
\end{center}
Leveraging these definitions and results, our subsequent aim is to establish a natural isomorphism between the identity functor $1_{\mathbb{OML}}$ and the composite functor $\Psi\circ \Gamma$.

\begin{theorem}\label{thm:tau-natural}
Let ${\mathcal{M}} = \left( {{M},{ \le },\mathop - \nolimits^{{ \bot }} } \right)$ be an orthomodular lattice, and let $\Gamma (\mathcal {M}) =\bigl(\mathscr{P}_{\text{fin}}(S_\mathcal{M}) ,\bigcup,\odot,\sim, { - ^*}\bigr)$ be its corresponding finitary orthomodular dynamic algebra. Furthermore, let $\delta : \mathcal{M} \to {\widetilde {\mathscr{P}_{\text{fin}}(S_\mathcal{M})}}$ be the established ortho-lattice isomorphism. Define a class function $\mu : 1_{\mathbb{OML}} \to \Psi\circ \Gamma$ such that $\mu_{\mathcal M} =\delta$ for every orthomodular lattice $\mathcal M$. Then, $\mu$ constitutes a natural isomorphism.
\end{theorem}

\begin{proof}
From Theorem~\ref{prop:chi-iso-finitary}, we know that $\mu_{\mathcal L}$ is a bijective $\mathbb{OML}$-morphism. Consider an orthomodular-lattice morphism $k : \mathcal M_1 \to \mathcal M_2$. We aim to demonstrate the commutativity of the following diagram:
\[
\begin{tikzcd}
\mathcal M_1 \ar[r,"k"] \ar[d,"\mu_{\mathcal M_1}"']
& \mathcal M_2 \ar[d,"\mu_{\mathcal M_2}"]\\
\Psi\bigl(\Gamma(\mathcal M_1)\bigr)\ar[r,"\Psi(\Gamma(k))"']
& \Psi\bigl(\Gamma(\mathcal M_2)\bigr)
\end{tikzcd}
\]
For any element $m \in M_1$, we can evaluate both paths in the diagram.
The path through the top and right yields:
$\bigl(\mu_{M_2}\circ k\bigr)(m) = \mu_{M_2}\bigl(k(m)\bigr)$. The path through the left and bottom yields:
\begin{center}
    $\bigl(\Psi(\Gamma(k))\circ\mu_{M_1}\bigr)(m) = \Psi\bigl(\Gamma(k)\bigr)\bigl(\{\pi_m\}\bigr) = \{\pi_{\,k(m)}\}$.
\end{center}
Since $\mu_{M_2}(k(m)) = \{\pi_{\,k(m)}\}$, it follows that $\bigl(\Psi(\Gamma(k))\circ\mu_{M_1}\bigr)(m) = \bigl(\mu_{M_2}\circ k\bigr)(m)$. As this holds for all $m \in M_1$, we conclude that $\Psi(\Gamma(k))\circ\mu_{M_1}=\mu_{M_2}\circ k$. Therefore, $\mu$ is a natural isomorphism.
\end{proof}


\subsubsection{The Natural Isomorphism $\lambda : 1_{\mathbb{FODA}} \to \Gamma \circ \Psi$}

Let $\mathfrak{K} = (K , {\bigsqcup } , {\odot } , \thicksim, {-^*} )$ be a finitary orthomodular dynamic algebra. Consequently, $\widetilde {\mathcal{K}} = (\widetilde K , {\preceq } , \thicksim )$ is an orthomodular lattice. By the definitions of the functors $\Psi$ and $\Gamma$, the following equality holds:
$$(\Gamma \circ \Psi)(K) = \Gamma(\Psi(K)) = \Gamma(\widetilde {K}) = \mathscr{P}_{\text{fin}}(S_{\widetilde {K}}).$$
Lemma~\ref{lemma2} states that every element $k \in K$ possesses a unique representation as $k = \bigsqcup\limits_{i=1}^m \bigl(p^{(i)}_1\odot p^{(i)}_2\cdots\odot p^{(i)}_{n_i}\bigr)$, where each $p^{(i)}_j\in \widetilde K$. For each finitary orthomodular dynamic algebra $\mathfrak{K}$, we define a map $\lambda_\mathfrak{K} : K \to \mathscr{P}_{\text{fin}}(S_{\widetilde {K}})$ for every $k = \bigsqcup\limits_{i=1}^m \bigl(p^{(i)}_1\odot p^{(i)}_2\cdots\odot p^{(i)}_{n_i}\bigr) \in K$ as:
$$\lambda_\mathfrak{K}(k) := \Bigl\{\,\pi_{p^{(i)}_1}\circ \pi_{p^{(i)}_2}\circ\cdots\circ \pi_{p^{(i)}_{n_i}} \;\Big|\; 1\le i\le m\Bigr\}.$$
Our subsequent objective is to demonstrate that $\lambda$ constitutes a natural isomorphism between the identity functor $1_{\mathbb{FODA}}$ and the composite functor $\Gamma \circ \Psi $.



\begin{theorem}\label{NI Lambda}
Let $\mathfrak{K} = (K , {\bigsqcup } , {\odot } , \thicksim, {-^*} )$ be a finitary orthomodular dynamic algebra. Let $\widetilde {\mathcal{K}} = (\widetilde K , {\preceq } , \thicksim )$ denote its associated orthomodular lattice. We define a class function $\lambda : 1_{\mathbb{FODA}} \to \Gamma\circ\Psi$ by specifying its components $\lambda_\mathfrak{K} : K \to \mathscr{P}_{\text{fin}}(S_{\tilde {\mathcal{K}}})$ for each $\mathfrak{K}$ as:
$$\lambda_\mathfrak{K}(k) = \Bigl\{\,\pi_{p^{(i)}_1}\circ \pi_{p^{(i)}_2}\circ\cdots\circ \pi_{p^{(i)}_{n_i}} \;\Big|\; 1\le i\le m\Bigr\}$$
for every element $k = \bigsqcup\limits_{i=1}^m \bigl(p^{(i)}_1\odot p^{(i)}_2\cdots\odot p^{(i)}_{n_i}\bigr) \in K$.
Then, $\lambda$ constitutes a natural isomorphism.
\end{theorem}

\begin{proof}
For any finitary orthomodular dynamic algebra $\mathfrak{K}$, the map $\lambda_\mathfrak{K}$ is bijective. This bijectivity arises from the uniqueness of the normal form for elements in $K$ and the surjectivity of the map onto all finite words of operations in $\mathscr{P}_{\text{fin}}(S_{\tilde {\mathcal{K}}})$. Moreover, we will show that $\lambda_\mathfrak{K}$ preserves the finitary join $\bigsqcup$, the multiplication $\odot$, the involution $*$, and the orthocomplement $\sim$.
To provide a more detailed proof that $\lambda_{\mathfrak{K}}$ preserves the algebraic operations of a finitary orthomodular dynamic algebra, we rely on the unique representation of elements $k \in K$ as finite joins of words in $\langle \widetilde{K} \rangle$. Let $\mathfrak{K} = (K, \bigsqcup, \odot, \sim, *)$ be a $\mathbb{FODA}$. For any $k \in K$, we write its unique normal form (per Lemma \ref{lemma2}) as:$$k = \bigsqcup_{i=1}^m w_i, \quad \text{where } w_i = p^{(i)}_1 \odot \cdots \odot p^{(i)}_{n_i} \text{ and } p^{(i)}_j \in \widetilde{K}.$$
The map is defined as $\lambda_{\mathfrak{K}}(k) = \{ \pi_{p^{(i)}_1} \circ \cdots \circ \pi_{p^{(i)}_{n_i}} \mid 1 \le i \le m \}$.

\begin{enumerate}
    \item Preservation of Finite Join ($\bigsqcup$): In $\Gamma(\Psi(\mathfrak{K}))$, the join operation is defined as the set-theoretic union $\cup$. Let $k = \bigsqcup_{i \in I} w_i$ and $l = \bigsqcup_{j \in J} v_j$ be elements in $K$ in their normal forms. By definition: $\lambda_{\mathfrak{K}}(k) = \{ \ulcorner w_i \urcorner \mid i \in I \}$ and $\lambda_{\mathfrak{K}}(l) = \{ \ulcorner v_j \urcorner \mid j \in J \}$. The join in $K$ is $k \sqcup l = (\bigsqcup_{i \in I} w_i) \sqcup (\bigsqcup_{j \in J} v_j)$. Applying $\lambda_{\mathfrak{K}}$:$$\lambda_{\mathfrak{K}}(k \sqcup l) = \{ \ulcorner w_i \urcorner \mid i \in I \} \cup \{ \ulcorner v_j \urcorner \mid j \in J \} = \lambda_{\mathfrak{K}}(k) \cup \lambda_{\mathfrak{K}}(l)$$. This matches the join operation in the target algebra $\Gamma(\Psi(\mathfrak{K}))$.

    \item Preservation of Multiplication ($\odot$): In $K$, multiplication distributes over joins (Definition \ref{deffoda}). Let $k = \bigsqcup_i w_i$ and $l = \bigsqcup_j v_j$.$$k \odot l = \left( \bigsqcup_i w_i \right) \odot \left( \bigsqcup_j v_j \right) = \bigsqcup_{i,j} (w_i \odot v_j)$$In $\Gamma(\Psi(\mathfrak{K}))$, the operation is defined as $A \odot B = \{ a \circ b \mid a \in A, b \in B \}$. So we have $\lambda_{\mathfrak{K}}(k \odot l) = \{ \ulcorner w_i \odot v_j \urcorner \mid \forall i, j \}$. By Lemma \ref{composition}, $\ulcorner w_i \odot v_j \urcorner = \ulcorner w_i \urcorner \circ \ulcorner v_j \urcorner$. Therefore $$\lambda_{\mathfrak{K}}(k \odot l) = \{ \ulcorner w_i \urcorner \circ \ulcorner v_j \urcorner \mid \forall i, j \} = \lambda_{\mathfrak{K}}(k) \odot \lambda_{\mathfrak{K}}(l)$$.

    \item Preservation of Involution ($*$): The involution in $K$ distributes over joins and reverses the order of multiplication. For $w_i = p_1 \odot \cdots \odot p_n$, we have $w_i^* = p_n^* \odot \cdots \odot p_1^*$. Note that $p_j \in \widetilde{K} \implies p_j^* \in \widetilde{K}$ (FODA 2). Then we have $\lambda_{\mathfrak{K}}(k^*) = \lambda_{\mathfrak{K}}(\bigsqcup w_i^*) = \{ \ulcorner w_i^* \urcorner \mid \forall i \}$. Using Lemma \ref{Sasaki is Involutive} and the definition of $\ulcorner \cdot \urcorner$, $\ulcorner w_i^* \urcorner = (\ulcorner w_i \urcorner)^*$. Thus $$\lambda_{\mathfrak{K}}(k^*) = \{ (\ulcorner w_i \urcorner)^* \mid \forall i \} = (\lambda_{\mathfrak{K}}(k))^*$$.

    \item Preservation of Orthocomplement ($\sim$): In $\Gamma(\Psi(\mathfrak{K}))$, the orthocomplement is defined as $\sim A = \{ \pi_{(\bigvee_{a \in A} a(1))^\perp} \}$. From Definition \ref{finitary-goda} and the associated constructions, the orthomodular lattice $\widetilde{K}$ has its join $\bigvee$ and complement $(\cdot)^\perp$ defined via $\sim$ and $\bigsqcup$ in $K$. Specifically, for $k = \bigsqcup w_i$, the "support" of $k$ in the lattice sense is related to the values of the operations applied to the identity. By Lemma \ref{lemma4} and Proposition \ref{prop:chi-iso-finitary}, the mapping $\delta(p) = \{ \pi_p \}$ is an isomorphism. In a $\mathbb{FODA}$, $\sim k$ is an element of $\widetilde{K}$. Since $\lambda_{\mathfrak{K}}$ maps $\sim k$ to the singleton containing the Sasaki projection of the lattice-complement of the "join of the heads" of the components of $k$: $$\lambda_{\mathfrak{K}}(\sim k) = \{ \pi_{(\bigvee_{i} \ulcorner w_i \urcorner(1))^\perp} \}$$. This coincides exactly with the definition of $\sim$ in $\Gamma(\Psi(\mathfrak{K}))$.
\end{enumerate}
These properties collectively establish that each $\lambda_\mathfrak{K}$ is a bijective $\mathbb{FODA}$–morphism.

Next, we demonstrate the naturality of $\lambda$. For any $\mathbb{FODA}$–morphism $\phi: \mathfrak{K}_1\to\mathfrak{K}_2$, we must show that the following diagram commutes:

\[
\begin{tikzcd}
\mathfrak{K}_1 \ar[r,"\lambda_{\mathfrak{K}_1}"] \ar[d,"\phi"'] &
  \Gamma(\Psi(\mathfrak{K}_1)) \ar[d,"\Gamma(\Psi(\phi))"] \\
\mathfrak{K}_2 \ar[r,"\lambda_{\mathfrak{K}_2}"'] &
  \Gamma(\Psi(\mathfrak{K}_2))
\end{tikzcd}
\]

Consider an element $k =\bigsqcup_{i=1}^m \bigl(p^{(i)}_1\cdots p^{(i)}_{n_i}\bigr) \in \mathfrak{K}_1$. By the definition of $\lambda_{\mathfrak{K}_1}$, we have $\lambda_{\mathfrak{K}_1}(k) = \Bigl\{\pi_{p^{(i)}_1}\circ\cdots\circ \pi_{p^{(i)}_{n_i}} \;\Big|\;1\le i\le m \Bigr\}$.

Since $\phi$ is an $\mathbb{FODA}$–morphism, it preserves the underlying orthomodular lattice structure and operations. Thus, $\phi$ maps each $p\in \widetilde{\mathfrak{K}_1}$ to $\phi(p)\in \widetilde{\mathfrak{K}_2}$ and preserves the operations $\bigsqcup$, $\odot$, $\sim$, and $-^*$. Consequently, $\phi(k)$ can be expressed as $\bigsqcup_{i=1}^m \bigl(\phi(p^{(i)}_1)\cdots\phi(p^{(i)}_{n_i})\bigr)$. Applying $\lambda_{\mathfrak{K}_2}$ to $\phi(k)$ yields $\lambda_{\mathfrak{K}_2}\bigl(\phi(k)\bigr) = \Bigl\{ \pi_{\phi(p^{(i)}_1)}\circ\cdots\circ \pi_{\phi(p^{(i)}_{n_i})} \;\Bigm|\;1\le i\le m \Bigr\}$.

Furthermore, by the definition of $\Gamma(\Psi(\phi))$, it holds that $\Gamma(\Psi(\phi))\bigl(\pi_p\bigr) = \pi_{\phi(p)}$. This property extends directly to compositions of operations, meaning $\Gamma(\Psi(\phi))\bigl(\pi_{p_1}\circ\cdots\circ \pi_{p_j}\bigr) = \pi_{\phi(p_1)}\circ\cdots\circ \pi_{\phi(p_j)}$.

Therefore, we definitively conclude that $\Gamma(\Psi(\phi))\bigl(\lambda_{\mathfrak{K}_1}(k)\bigr) = \lambda_{\mathfrak{K}_2}\bigl(\phi(k)\bigr)$, which verifies the commutativity of the diagram and, thus, the naturality of $\lambda$.

\end{proof}

The foregoing analysis leads to the following conclusions.

\begin{theorem}
    The quadruple $\left( {\Gamma ,\Psi ,\mu ,\lambda } \right)$ establishes a categorical equivalence between $\mathbb{OML}$ and $\mathbb{FODA}$.
\end{theorem}

\begin{proof}
    Theorem \ref{thm:tau-natural} demonstrates that $\mu$ is a natural transformation from $1_{\mathbb{OML}}$ to $\Psi \circ \Gamma$. Similarly, Theorem \ref{NI Lambda} establishes $\lambda$ as a natural transformation from $1_{\mathbb{FODA}}$ to $\Gamma \circ \Psi$. Consequently, these properties collectively demonstrate that $\left( {\Gamma ,\Psi ,\mu ,\lambda } \right)$ constitutes a categorical equivalence between $\mathbb{OML}$ and $\mathbb{FODA}$.
\end{proof}

\section{CONCLUSION}
Our research demonstrates a categorical equivalence between orthomodular lattices and finitary orthomudular dynamic algebras, established through a newly designed functor. For future work, we aim to extend the finitary orthomodular dynamic algebra to the concepts of F-ODA and R-ODA (as described in \cite{Soroush1954}) and investigate their categorical equivalence.

\section{ACKNOWLEDGEMENT}
The authors were supported by the project MUNI/A/1457/2023 by Masaryk
University. The authors is indebted to the reviewer for his/her thorough comments.


\end{document}